\documentclass[11pt,a4paper]{amsart}
\usepackage{hyperref}
\usepackage{amssymb, amsthm, amsmath}
\usepackage{enumerate}
\usepackage[english]{babel}
\usepackage[all]{xy}
\usepackage[usenames]{color}

\theoremstyle{theorem}
\newtheorem{theorem}{Theorem}[section]

\theoremstyle{corollary}
\newtheorem{corollary}[theorem]{Corollary}

\theoremstyle{lemma}
\newtheorem{lemma}[theorem]{Lemma}

\theoremstyle{example}
\newtheorem{example}[theorem]{Example}

\theoremstyle{definition}
\newtheorem{definition}[theorem]{Definition}

\theoremstyle{remark}
\newtheorem{remark}[theorem]{Remark}

\theoremstyle{ques}
\newtheorem{ques}[theorem]{Question}

\theoremstyle{proposition}
\newtheorem{proposition}[theorem]{Proposition}

\theoremstyle{claim}
\newtheorem{claim}[theorem]{Claim}

\newenvironment{Example}{\begin{example}\rm}{\end{example}}

\def\id{\mbox{Id} }
\def\et{\quad\mbox{and}\quad}

\DeclareMathOperator{\GL}{GL}
\DeclareMathOperator{\SL}{SL}

\DeclareMathOperator{\PSL}{PSL}
\DeclareMathOperator{\Sp}{Sp}

\DeclareMathOperator{\pr}{pr}

\DeclareMathOperator{\rank}{rank}

\DeclareMathOperator{\Ad}{Ad}
\DeclareMathOperator{\codim}{codim}
\DeclareMathOperator{\cone}{cone}
\DeclareMathOperator{\Cl}{Cl}
\DeclareMathOperator{\Cent}{C}
\DeclareMathOperator{\Center}{Z}
\DeclareMathOperator{\op}{op}
\DeclareMathOperator{\tr}{tr}
\DeclareMathOperator{\Gal}{Gal}
\DeclareMathOperator{\Aut}{Aut}
\DeclareMathOperator{\SAut}{SAut}
\DeclareMathOperator{\Pic}{Pic}

\renewcommand{\id}{\textrm{id}}
\newcommand{\Lie}[1]{\mathfrak{#1}}

\newcommand{\UUU}{\mathcal{U}}
\newcommand{\NNN}{\mathcal{N}}
\newcommand{\OOO}{\mathcal{O}}

\def\AA{\mathbb{C}}
\def\PP{\mathbb{P}}
\def\C{\mathbb{C}}
\def\CC{\mathbb{C}}
\def\ZZ{\mathbb{Z}}

\def\R{\mathbb{R}}
\newcommand{\aquot}{/ \hspace{-2.5pt} /}

\begin{document}

\title[Uniqueness of Embeddings of $\AA$ into Algebraic Groups]
{Uniqueness of Embeddings of the Affine Line into Algebraic Groups}
\author{Peter Feller and Immanuel van Santen n\'e Stampfli}
\address{Max Planck Institute for Mathematics, Vivatsgasse 7, 53111 Bonn, Germany}
\email{peter.feller@math.ch}
\address{Fachbereich Mathematik der Universit\"at Hamburg,
Bundesstra{\ss}e 55, 20146 Hamburg, Germany}
\email{immanuel.van.santen@math.ch}


\begin{abstract}
	Let $Y$ be the underlying variety of a connected affine algebraic group.
	We prove that two embeddings of the affine line $\AA$
	into $Y$ are the same up to an automorphism of $Y$
	provided that $Y$ is not isomorphic to a product of a
	torus $(\CC^\ast)^k$ and one of the three varieties
	$\AA^3$, $\SL_2$, and~$\PSL_2$.
\end{abstract}

\maketitle

\section{Introduction}

In this paper, \emph{varieties} are understood to be (reduced) algebraic
varieties over the field of complex numbers $\C$, carrying the Zariski topology.
We say that two
closed\footnote{All embeddings in this paper are closed. In fact,
closedness is implied for embeddings of the affine line $\C$ into quasi-affine varieties which is the setting we are considering.} embeddings of varieties
$f,g\colon X\to Y$
are \emph{equivalent} or \emph{the same up to an automorphism of $Y$}
if there exists an automorphism $\varphi \colon Y\to Y$ such that $\varphi\circ f=g$. We consider embeddings of the affine line $\C$ into varieties $Y$ that arise as underlying varieties of affine algebraic groups
and study these embeddings up to automorphisms of $Y$. Recall that an
\emph{affine algebraic group} is a closed subgroup of
the complex general linear group~$\GL_n$ for some $n$.
In this paper, all groups are affine and algebraic.
Our main result is the following.


\begin{theorem}\label{thm:mainthm}
	{Let $Y$ be the underlying variety of a connected affine algebraic
	group. Then two
	embeddings of the affine line $\AA$ into
	$Y$ are the same up to an automorphism
	of $Y$ provided that $Y$ is not isomorphic to a product of a
	torus $(\CC^\ast)^k$ and one of the three varieties
	$\AA^3$, $\SL_2$, and~$\PSL_2$.}

\end{theorem}
In particular, $\AA$ embeds uniquely (up to automorphisms)
into affine algebraic
groups without non-trivial characters of dimension different than~$3$.
Note also that connectedness is not a restriction since any connected component of an algebraic group $G$ is itself isomorphic (as a variety) to the connected component of the identity element.

Let us put Theorem~\ref{thm:mainthm} in context.
Embedding problems are most classically considered for $Y = \CC^n$; compare e.g.~the overviews by
Kraft and
van den Essen~\cite{Kr1996Challenging-proble, Es2004Around-the-Abhyank}.
We recall what is known about uniqueness of embeddings of $\C$ into $\C^n$.
If $n=2$, there is a unique embedding (up to automorphisms)
by the Abhyankar-Moh-Suzuki
Theorem~\cite{AbMo1975Embeddings-of-the-, Su1974Proprietes-topolog}. For
$n\geq 4$, again there is a unique embedding
(up to automorphisms) by the
work of Jelonek; see~\cite{Je1987The-extension-of-r}.
More generally, Kaliman \cite{Ka1991Extensions-of-isom}
and Srinivas~\cite{Sr1991On-the-embedding-d} proved that smooth
varieties of dimension $d$ embed uniquely into $\C^n$ whenever $n\geq 2d+2$.
The existence of non-equivalent embeddings $\CC \to \CC^3$ is a
long standing open problem; see~\cite{Kr1996Challenging-proble}.
There are various potential examples of non-equivalent embeddings of
$\CC$ into $\CC^3$; see e.g.~\cite{Sh1992Polynomial-represe}.

Srinivas' result is established by cleverly projecting to different linear coordinates.
The second author was able to use projections to coordinates to establish that there is a unique embedding of $\C$ into $\SL_n$ (up to automorphisms) for all integers $n\geq 3$; see \cite{St2015Algebraic-Embeddin}. For algebraic groups in general, projections to coordinates are no longer available. Our approach to embeddings of $\C$ is to study projections onto quotients by unipotent subgroups. 

For a different point of view we consider the notion of flexible varieties as studied by Arzhantsev, Flenner, Kaliman, Kutzschebauch, and
Zaidenberg in~\cite{ArFlKa2013Flexible-varieties}. Flexible varieties can be seen as generalization of connected affine algebraic groups without non-trivial characters.
Smooth irreducible affine flexible varieties of dimension at least $2$ have the property that all embeddings of a fixed finite set are equivalent~\cite[Theorem~0.1]{ArFlKa2013Flexible-varieties}. Theorem~\ref{thm:mainthm} states that in most affine algebraic groups even all embeddings of $\C$ are equivalent.
In light of Theorem~\ref{thm:mainthm}, the following question is
natural in this context.
\begin{ques}\label{qu:felxvarhaveuniqueembeddings}
Let $Y$ be a smooth irreducible affine
flexible variety of dimension at least $4$.
Is there at most one embedding of $\AA$ into $Y$ up to automorphisms?
\end{ques}

There exist smooth irreducible flexible affine surfaces
that contain non-equivalent embeddings of $\CC$;
see Example~\ref{exa.Dimension2}. Since in dimension three
there is the long standing open problem, whether all
embeddings of $\CC$ into $\CC^3$ are equivalent, we
ask Question~\ref{qu:felxvarhaveuniqueembeddings} only for varieties of
dimension $\geq 4$.
In Example~\ref{exa.Dimension3},
we provide a contractible smooth affine irreducible surface $S$ such that
$S\times \AA^n$ contains non-equivalent embeddings of
$\AA$ for all integers $n \geq 1$.
These examples of varieties that contain non-equivalent embeddings
of $\CC$ are the content of Section~\ref{Sec.Non-Equivalent-Examples}.

Note that some sort of `flexibility' is required to prove results such as
Theorem~\ref{thm:mainthm} in case one has `many' embeddings of $\CC$.
For example, if every pair of points in an affine variety $Y$ can be
connected by a chain of embedded affine lines\footnote{Compare
the notion of $\mathbb{A}^1$-chain connectedness in
\cite{AsMo2011Smooth-varieties-u} and
rationally chain connectedness in \cite{Ko1996Rational-curves-on}.}
 and $Y$ admits a non-trivial
$\CC^+$-action, then flexibility of $Y$ is a necessary condition for
the equivalence of all embeddings $\CC \to Y$.

Theorem~\ref{thm:mainthm} can be seen as covering all cases of embeddings of
$\AA$ into connected affine algebraic groups without non-trivial characters except the well-known open problem of embeddings into $\AA^3$ and embeddings into $\SL_2$ and $\PSL_2$. As argued by the second author 
in~\cite{St2015Algebraic-Embeddin},
$\SL_2$ (and in fact similarly $\PSL_2$) allows for many embeddings of $\CC$
and perceivably their equivalence or non-equivalence up to
automorphism might be as challenging as for the $\AA^3$ case.
In Section~\ref{sec:embinto3dgroups}, we report on these examples
of embeddings into $\CC^3$, $\SL_2$ and $\PSL_2$.


\subsection{Tools for the proof of Theorem~\ref{thm:mainthm}}

In Section~\ref{sec.notions}, notions and basic facts
from the theory of algebraic groups and there principal bundles are introduced.

In order to prove equivalence of embeddings we need a good
way to construct automorphisms. This is the content of Section~\ref{sec.Auto}.
Let us expand on that.
While we are only interested in showing uniqueness of embeddings
up to automorphisms of the \emph{underlying variety} of
an algebraic group, we will heavily depend on the group structure to construct automorphisms. The following shearing-tool
follows readily by using the group structure;
see Proposition~\ref{prop.ConstructionOfAuto}. It is our main
tool to construct automorphisms of the underlying variety of an algebraic
group.

\begin{quote}
\textbf{Shearing-tool.}
Let $X$ and $X'$ be affine lines embedded in an algebraic group
$G$ and let $H \subseteq G$ be a closed subgroup such that $G/H$
is quasi-affine. If $\pi \colon G \to G/H$ restricts to an embedding on
$X$ and $X'$
and if $\pi(X)=\pi(X')$, then there exists a $\pi$-fiber-preserving
automorphism of $G$ mapping $X$ to $X'$.
\end{quote}

This could be seen as an analog to a fact used in proving Srinivas'
result about embeddings into $\C^n$: given two embeddings
$\sigma,\sigma'$ of an affine line (or in fact any affine variety) into
$\C^n$ such that the last $m<n$ coordinate functions agree and yield
an embedding into $\C^m$, then there exists a shear $\phi$ of $\C^n$
with respect to the projection to the last
$m$ coordinates such that $\phi\circ \sigma=\sigma'$; see
\cite{Sr1991On-the-embedding-d}.

In Section~\ref{sec.embedwithunipotentimage}, we show that all embeddings
$\C \to G$ with image a one-dimensional unipotent subgroup of $G$
are equivalent. Thus in order
to prove equivalence of all embeddings $\CC \to G$, it suffices
to show that every affine line in $G$
can be moved via an automorphism of $G$ into a
one-dimensional unipotent subgroup.

In Section~\ref{sec.genproj}, we introduce another tool.
In view of the above shearing-tool, given a curve $X$ in $G$,
we are interested in having many closed subgroups $H$ such that $X$
projects isomorphically (or at least birationally) to $G/H$. We establish several
results in that direction and we call them generic projection results.
In this context our main result is the following;
see Proposition~\ref{prop.UnipotProj}. It is based on an elegant
formula that relates the dimension of the conjugacy class $C$ of a
unipotent element in a semisimple group with the dimension
of the intersection of $C$ with a maximal unipotent subgroup;
see \cite{St1976On-the-desingulari}
and \cite[\S6.7]{Hu1995Conjugacy-classes-}.

\begin{quote}
\textbf{Main generic projection result.}
	If $G$ is a simple algebraic group of rank at least two,
	and $H$ a closed unipotent subgroup, then for any
	curve $X \subseteq G$ that is isomorphic to $\CC$
	there exists an automorphism $\varphi$ of $G$ such that for generic
	$g \in G$ the quotient map $G \to  G / g H g^{-1}$
	restricts to an embedding on $\varphi(X)$.
\end{quote}

\subsection{Outline of the proof of Theorem~\ref{thm:mainthm}}

In Section~\ref{sec.ReductionSemisimple}, we
reduce Theorem~\ref{thm:mainthm} to the case of a semisimple group.
In a bit more detail: let $G$ be an algebraic group satisfying the assumptions
of Theorem~\ref{thm:mainthm}.
We note that $G$ is isomorphic
(as a variety) to $G^u\times (\C^*)^k$ for some integer $k \geq 0$,
where $G^u$ denotes the normal subgroup of $G$ generated by
unipotent elements. Embeddings of $\C$ into $G^u\times (\C^*)^n$ are necessarily
constant on the second factor; thus we study embeddings into $G^u$. We have
that $G^u$ is isomorphic as a variety to
$R_u(G^u) \times G^u/R_u(G^u)$, where $R_u(G^u)$ denotes the
unipotent radical---the largest normal unipotent subgroup of $G^u$.
If $R_u(G^u)$ is non-trivial nor equal to $G^u$, then the non-trivial product
structure on $G^u$ allows
to show equivalence of all embedded affine lines;
see Proposition~\ref{prop.Product}.
If $R_u(G^u)=G^u$, then $G^u \cong \C^n$ for some $n\neq 3$, and the
result follows by Jelonek's work (for $n\geq 4$) and by the
Abhyankar-Moh-Suzuki Theorem (for $n=2$). This leaves the case where
$R_u(G^u)$ is trivial, i.e.~$G^u$ is semisimple.

In Section~\ref{sec.ReductionSimple}, we prove
Theorem~\ref{thm:mainthm} in case of a semisimple, but not simple group $G$.
We use that $G$ is isomorphic to a quotient of the product of at least two
simple groups
by a finite central subgroup.
Part of the argument relies on the fact that simple groups have sufficiently
many unipotent elements. To ensure this, the classification of simple groups
of small rank is invoked; see Lemma~\ref{lem.enoughUnipotentsInAParabolic}.

Finally, in Section~\ref{sec.Simple}, we prove
Theorem~\ref{thm:mainthm} in the case of a simple
group $G$. This constitutes the technical heart of the proof. Besides using several results from previous sections about embeddings into products and generic projection results, we use the language of algebraic group theory to define an interesting subvariety $E$ of $G$. In fact, $E$ is the preimage of the (unique) Schubert curve under the projection to $G/P$, where $P$ is a maximal parabolic subgroup of $G$. We show that any embedding of the affine line in $G$ can be moved into $E$ by an automorphism of $G$; compare Subsection~\ref{sec.MovingIntoE}. This is in fact the key step in our proof.
Let us expand on this.

{	
	Let $P^-$ be an opposite parabolic subgroup to $P$ and denote
	by $\pi \colon G \to G/ R_u(P^-)$ the quotient map.
	We establish, that the restriction of $\pi$ to $E$ is a locally
	trivial $\CC$-bundle over $\pi(E)$
	and $\pi(E)$ is a big open subset of $G/ R_u(P^-)$,
	i.e.~the complement is a closed subset of codimension at least two in
	$G/ R_u(P^-)$; see Proposition~\ref{prop.keyProperteOfPi}.
	Now, one can move $X$ into $E$ via the following steps.
	\begin{itemize}
		\item Using our main generic projection result, we can achieve
			that $\pi$ restricts to an embedding on $X$.
		\item Using that $\pi(E)$ is a big open subset of $G/R_u(P^-)$,
			 we can move $X$ into $\pi^{-1}(\pi(E))$ by left multiplication
			 with a group element. In particular,
			 $\pi$ still restricts to an embedding on $X$,
			 by $G$-equivariancy.
		\item Since $E \to \pi(E)$ is a locally trivial $\CC$-bundle, it
			has a section $X' \subseteq E$ over
			$\pi(X) \cong \CC$. Therefore,
			we can move $X$ into $X'$ with our shearing-tool.
	\end{itemize} }

Next we exploit that $E=KP$ for a certain non-trivial closed subgroup
$K$ of $G$ and the parabolic subgroup
$P$ used to define $E$. Under the assumption that the rank of
$G$ is at least two, i.e. $G$ is different from $\SL_2$ and $\PSL_2$, we show the
following. Via an automorphism
of $G$ one can move
any affine line in $E$ to an affine line in $E$ such that the quotient map
$E \to K \backslash E$ restricts to an embedding on this affine line;
see Proposition~\ref{prop.key}.
Using this result and the fact that the product
map $K \times P \to E$ is a principal $K\cap P$-bundle we can move
any affine line in $E$ into an affine line in $P$. Since $P$
is a proper subgroup of $G$, one can move any affine line in $P$
into a one-dimensional unipotent subgroup of $G$.
This implies Theorem~\ref{thm:mainthm} in this last case.

\subsection{Overview of the appendices}

We have three appendix sections, which contain results that are used in the proof of Theorem~\ref{thm:mainthm}, but that are either classical or the proofs are independent of the general idea of the proof of Theorem~\ref{thm:mainthm}. Appendix~\ref{sec.PrincipalBundlesOverA1} provides a proof of the fact that principal $G$-bundles over the affine line are trivial for all affine algebraic
groups $G$. In Appendix~\ref{sec.Genonparabsubgroups}, we provide generalities on parabolic subgroups of reductive groups and the dimension of their
subvariety of unipotent elements
and their unipotent radical as needed in Section~\ref{sec.Simple}.
In Appendix~\ref{sec.C-equivmorphofsurf}, we provide results about
$\CC^+$-equivariant morphisms of surfaces
as needed in the proof of Proposition~\ref{prop.key}
(which constitutes the most technical part of Section~\ref{sec.Simple}).

\subsection*{Acknowledgments}
We thank J\"org Winkelmann for informing us about
Lemma~\ref{lem.OneDimProj}, Adrien Dubouloz
for many discussions concerning the non-equivalent embeddings
in Section~\ref{Sec.Non-Equivalent-Examples}, and Julie Decaup and
Adrien Dubouloz for allowing us to insert Example~\ref{exa.Dimension2}.

\section{Examples of varieties that contain non-equivalent embeddings of $\AA$}\label{sec.examplesofnonequivembed}
\label{Sec.Non-Equivalent-Examples}

In the first example we provide an irreducible smooth
affine flexible surface
that contains non-equivalent embeddings of $\AA$.
This example is due to Decaup and Dubouloz. For a deeper study
of this example see \cite{DeDu2016Affine-Lines-in-th}.

\begin{Example}
\label{exa.Dimension2}
{
Let $S = \PP^2 \setminus Q$, where $Q$
is a smooth conic in $\PP^2$. Clearly, $S$ is irreducible, smooth and affine.
Let $(x: y: z)$ be a homogeneous coordinate system of $\PP^2$.
We can assume without loss of generality that $Q$ is given by the
homogeneous equation $xz = y^2$ in $\PP^2$.
}

Let $L_1$ be the curve $S \cap \{ z = 0\}$ and let $L_2$ be the curve
$S \cap \{xz-y^2 = z^2\}$. One can see that $\Pic(S \setminus L_1)$ is trivial,
whereas $\Pic(S \setminus L_2)$ is isomorphic to $\ZZ / 2 \ZZ$. Hence
there are non-equivalent embeddings of $\CC \cong L_1 \cong L_2$ into $S$.

To establish the flexibility of $S$, we have to show that
$\SAut(S)$ acts transitively on $S$ where
$\SAut(S)$ denotes the subgroup of $\Aut(S)$ that
is generated by all automorphisms coming from $\CC^+$-actions on $S$;
see~\cite[Theorem~0.1]{ArFlKa2013Flexible-varieties}.
Consider the $\CC^+$-action
$t \cdot (x : y: z) = (x: y + tx : z + 2yt+t^2 x)$ on $S$.
A computation shows that every orbit of this $\CC^+$-action
intersects the curve $L_2$. Since $L_2$ is an orbit of the
$\CC^+$-action $t \cdot (x : y: z) = (x+2yt+t^2z: y + tz : z)$ on $S$,
it follows that $\SAut(S)$ acts transitively on $S$.

\end{Example}

	Next, we give in any dimension $\geq 3$
	an example of an irreducible
	smooth contractible affine variety that contains
	non-equivalent embeddings of $\CC$.
	Note that for any irreducible smooth contractible affine variety,
	the ring of regular functions
	is a unique factorization domain and all invertible functions on it are constant;
	see e.g.~\cite[Proposition~3.2]{Ka1994Exotic-analytic-st}.
	
\begin{Example}
	\label{exa.Dimension3}
	Let $S$ be an irreducible smooth
	contractible affine surface of logarithmic Kodaira
	dimension one that contains a copy $C$ of the affine line.
	For example, by \cite[Theorem~A]{DiPe1990Contractible-affin}
	the affine hypersurface in $\AA^3$ defined by
	\[
		z^2 x^3 + 3zx^2 + 3x -zy^2-2y = 1
	\]
	is smooth, contractible and of logarithmic Kodaira dimension one,
	and $z=0$ inside this hypersurface defines a
	copy of $\AA$. Since $S$ is smooth, affine and of
	logarithmic Kodaira dimension one, there exists no ${\C^+}$-action on $S$, by
	\cite[Lemma~1.3]{MiSu1980Affine-surfaces-co}.
	In other words, the Makar-Limanov invariant of
	$S$ is equal to the ring of regular functions on $S$.
	Now, by \cite[Corollary 5.20]{Cr2004On-the-AK-invarian},
	it follows that the Makar-Limanov invariant of
	\[
		S \times \AA^n
	\]
	is equal to the ring of regular functions on $S$. In particular,
	every automorphism of $S \times \AA^n$ maps fibers of the canonical
	projection $\pi \colon S \times \AA^n \to S$ to fibers of it. Thus any copy
	of $\AA$ inside $S \times \AA^n$ that lies in some fiber of $\pi$
	is non-equivalent to the section $C \times \{ 0 \} \subseteq S \times \AA^n$
	of $\pi$ over $C$.
	In summary, we proved that $S \times \AA^n$
	is irreducible, affine, smooth, contractible and
	contains non-equivalent copies of $\AA$, provided
	that $n \geq 1$.
\end{Example}

To compare Example~\ref{exa.Dimension2} and Example~\ref{exa.Dimension3},
note that there exists no smooth irreducible
affine surface that is \emph{contractible}
and contains two non-equivalent copies of $\AA$.
Indeed, smooth homology planes of logarithmic Kodaira dimension one or two,
contain at most one copy of $\AA$ and smooth homology
planes of logarithmic Kodaira dimension zero do not exist; see e.g.~\cite{GuMi1992Affine-lines-on-lo}. If the logarithmic Kodaira dimension
of a smooth, contractible affine surface is
$-\infty$, then it must be $\AA^2$ by
Miyanishi's characterization of the affine plane; see~\cite{Mi1975An-algebraic-chara} and~\cite{Mi1984An-algebro-topolog}. Thus,
the {Abhyankar-Moh-Suzuki Theorem} implies our claim.

\section{\texorpdfstring{Examples of embeddings of $\AA$ into
$\AA^3$, $\SL_2$ and $\PSL_2$}
{Examples of embeddings of C into
C3, SL2 and PSL2}}

\label{sec:embinto3dgroups}
In this section we discuss what is known about embeddings of
$\C$ into $\C^3$ and give embeddings of $\C$ into $\SL_2$ and
$\PSL_2$ arising from embeddings of $\C$ into $\C^3$. 

\subsection{Embeddings into $\CC^3$}
After Abyankar and Moh and, independently, Suzuki established uniqueness of embeddings of $\AA$ into $\AA^3$, many examples of embeddings of $\AA$ into $\AA^3$ that are potentially different (up to automorphisms) from the standard embedding $\AA\to\AA^3$, $t \mapsto (t,0,0)$ where suggested; many of these have since been proven to be standard; compare e.g.~\cite{Es2004Around-the-Abhyank}. However, examples due to Shastri, which are based on the idea of using embeddings with real coefficients such that the restriction map $\R\to\R^3$ is knotted, seem among the most promising to be non-standard. Concretely, the embeddings $\AA\to\AA^3$
\[
	t\mapsto (t^3-3t,t^4-4t^2,t^5-10t) \et t\mapsto (t^3-3t,t(t^2-1)(t^2-4),t^7-42t) \, ,
\]
which restrict to embeddings $\R\to\R^3$ of a trefoil knot and a figure eight knot, respectively, are not known to be standard; see~\cite{Sh1992Polynomial-represe}.

\subsection{\texorpdfstring{Comparison of embeddings into $\CC^3$ and $\SL_2$}{Comparison of embeddings into C3 and SL2}}
\label{subsec.ComparisonC3andSL2}
Embeddings of $\AA$ into $\SL_2$ are less studied.
Following an example of the second author
(compare~\cite{St2015Algebraic-Embeddin}), we briefly discuss how embeddings into $\AA^3$ give rise to embeddings into $\SL_2$.
In fact, for any embedding $h$ of $\AA$ into $\C^3$ there exists an
automorphism $\varphi$ of $\C^3$ such that
\begin{equation}\label{eq:embintoSL2}t\mapsto
		\begin{pmatrix}
                  	f_1(t) & (f_1(t)f_3(t)-1) / f_2(t) \\
                  	f_2(t) & f_3(t)
                \end{pmatrix}
\end{equation}
defines an embedding of $\AA$ into $\SL_2$
where $f_1$, $f_2$ and $f_3$ are the components of $f = \varphi \circ h$.
In fact, it suffices to arrange that $f_2$ divides $f_1f_3-1$
in $\C[t]$, which is explicitly done in~\cite{St2015Algebraic-Embeddin}.

On the other hand, if we start with an embedding $g$ of $\CC$ into $\SL_2$,
then there exists an automorphism $\psi$ of $\SL_2$ such that
$p \circ \psi \circ g$ is an embedding of $\CC$ into $\CC^3$ where
$p \colon \SL_2 \to \CC^3$ is the projection to three coordinate functions
of $\SL_2$; see \cite[Lemma~10]{St2015Algebraic-Embeddin}.

\subsection{\texorpdfstring{Comparison of embeddings into
$\SL_2$ and $\PSL_2$}{Comparison of embeddings into SL2 and PSL2}}
	
In this subsection we construct a natural surjective map
from the set of all embeddings of $\CC$ into $\PSL_2$
to the set of all embeddings of $\CC$
into $\SL_2$ where we consider the embeddings up to automorphisms.
Thus, using Subsection~\ref{subsec.ComparisonC3andSL2}, every embedding
of $\CC$ into $\CC^3$ gives rise to an embedding of $\CC$ into $\PSL_2$.

By Hurwitz's Theorem,
every finite \'etale morphism $E \to \CC$ is trivial in the sense
that every connected component of $E$ maps isomorphically onto $\CC$;
see e.g.~\cite[Chp. IV, Corollary~2.4]{Ha1977Algebraic-geometry}).
In particular, every embedding of $\CC$ into $\PSL_2$ lifts
via the canonical quotient $\eta \colon \SL_2 \to \PSL_2$
to two embeddings into $\SL_2$, which are the same up to
the involution $X \mapsto -X$ of $\SL_2$.
Since every automorphism of $\PSL_2$ lifts to an automorphism
of $\SL_2$ via $\eta$
(see~\cite[Proposition~20]{Se1958Espaces-fibres-}), we constructed a well-defined map
\begin{align*}
	\Xi \colon & \{ \, \textrm{Embeddings of $\CC$ into $\PSL_2$ up
	automorphisms of $\PSL_2$} \, \} \to \\
	& \{ \, \textrm{Embeddings of $\CC$ into $\SL_2$ up to
	automorphisms of $\SL_2$} \, \} \, .
\end{align*}
We claim that $\Xi$ is surjective. For this, let $f \colon \CC \to \SL_2$ be
an embedding. It is enough to prove that there exists
an automorphism $\varphi$ of $\SL_2$
such that $\eta \circ \varphi \circ f$ is an embedding into $\PSL_2$. Since
$\eta \circ \varphi \circ f$ is always immersive and proper,
we only have to prove injectivity of $\eta \circ \varphi \circ f$.
Let $\pi_i \colon \SL_2 \to \CC^2 \setminus \{ 0 \}$
be the projection to the $i$-th column.
We can assume, after composing $f$
with an automorphism of $\SL_2$,
that $\pi_1 \circ fÊ\colon \CC \to \CC^2 \setminus \{ 0 \}$
is immersive; see
\cite[Lemma~10]{St2015Algebraic-Embeddin}. Let $C$
be the image of $\pi_1 \circ f$, which is closed in $\CC^2 \setminus \{Ê0 \}$.
There is a commutative diagram
\[
	\xymatrix{
		\SL_2 \ar[d]_-{\eta} \ar[r]^-{\pi_1} & \CC^2 \setminus \{ 0Ê\} \ar[d]^-{\rho} \\
		\PSL_2 \ar[r] & V
	}
\]
where $\rho \colon \CC^2 \setminus \{Ê0 \} \to V$ denotes the quotient by
the $\ZZ / 2 \ZZ$-action $(x, z) \mapsto (-x, -z)$ on $\CC^2 \setminus \{ 0 \}$. Let $Z = \rho(C)$.
Since the morphism $\rho$ is \'etale, it follows that
$\rho \circ \pi_1 \circ f \colon \C \to Z$ is immersive and hence birational.
Let $Z_0 \subseteq Z$ be a finite subset such that $\rho \circ \pi_1 \circ f$
restricts to an isomorphism
$\CC \setminus (\rho \circ \pi_1 \circ f)^{-1}(Z_0) \cong Z \setminus Z_0$.
Let $T$ be the finite set $(\rho \circ \pi_1 \circ f)^{-1}(Z_0)$.
There exists a morphism $p \colon \CC^2 \to \CC$ such that for all $t \neq s$ in $T$
we have
\begin{align}
	 \label{eq.Comparison}
	 & (\pi_2 \circ f)(t) + p((\pi_1 \circ f)(t)) \cdot (\pi_1 \circ f)(t) \neq \\
	\notag
	& \pm \left[ (\pi_2 \circ f)(s) + p((\pi_1 \circ f)(s)) \cdot (\pi_1 \circ f)(s) \right] \, .
\end{align}
Indeed, such a $p$ exists, since for all $t \neq s$ in $T$
the negation of condition~\eqref{eq.Comparison} defines two non-trivial
affine linear equations for $p$
in the vector space of functions $\CC^2 \to \CC$.
Let $\varphi \colon \SL_2 \to \SL_2$ be the automorphism
given by
\[
	\varphi
	\begin{pmatrix}
		x & y \\
		z & w
	\end{pmatrix} =
	\begin{pmatrix}
		x & y + xp(x, z) \\
		z & w + zp(x, z)
	\end{pmatrix} \, .
\]
Let $g = \varphi \circ f$.
Note that $\pi_1 \circ g = \pi_1 \circ f$ and
$\pi_2 \circ g = \pi_2 \circ f + (p \circ \pi_1 \circ f) \cdot (\pi_1 \circ f)$.
Since $\rho \circ \pi_1 \circ f$ restricts to an isomorphism
$\CC \setminus T \cong Z \setminus Z_0$,
it follows that $\eta \circ g$ restricted to $\CC \setminus T$ is injective.
By~\eqref{eq.Comparison}, we have
$(\eta \circ g)(t) \neq (\eta \circ g)(s)$ for all $t \neq s$ in $T$
and thus $\eta \circ g$ restricted to $T$ is injective.
Since the images under $\eta \circ g$
of $\CC \setminus T$ and $T$ are disjoint, it follows that
$\eta \circ g$ is injective, which implies our claim.


\section{Notation and generalities on algebraic groups and their principal bundles}
\label{sec.notions}
\subsection{Algebraic groups}

For the basic results on algebraic groups we refer to
\cite{Hu1975Linear-algebraic-g}
and for the basic results about Lie algebras and root systems
we refer to \cite{Hu1978Introduction-to-Li}.
In order to set up conventions, let us recall the basic terms.
A connected non-trivial algebraic group $G$ is called \emph{semisimple}
if it has a trivial radical $R(G)$, where $R(G)$
is the largest connected normal solvable subgroup of $G$.
An algebraic group $G$ is called \emph{reductive} if
it has a trivial unipotent radical $R_u(G)$, where $R_u(G)$
is the closed normal subgroup of $R(G)$ consisting
of all unipotent elements.
A non-commutative connected algebraic group $G$
is called $\emph{simple}$, if it contains
no non-trivial closed connected normal
subgroup. Note that for a simple algebraic group
$G$, the quotient $G/\Center(G)$ by the center $\Center(G)$
is simple as an abstract group
(see \cite[Corollary~29.5]{Hu1975Linear-algebraic-g}),
i.e.~it contains no proper normal subgroup.

For any connected algebraic group $G$, we denote by $\UUU_G$ the subset
of unipotent elements in $G$. It is irreducible and closed in $G$;
see~\cite[Theorem~4.2]{Hu1995Conjugacy-classes-}. We denote
by $\rank(G)$ the dimension of a maximal torus of $G$. By
\cite[\S4.2]{Hu1995Conjugacy-classes-} we have for any reductive
group $G$
\[
	\dim \UUU_G = \dim G - \rank G
\]
and using the Levi decomposition
(see \cite[Theorem~4, Chp.~6]{OnVi1990Lie-groups-and-alg})
this formula holds more generally for every connected algebraic group $G$.
Moreover, we denote by $G^u$ the normal subgroup which is generated by
all unipotent elements of $G$. It is connected and closed in $G$; see
\cite[Proposition~7.5]{Hu1975Linear-algebraic-g}.
For any semisimple $G$, we have $G = G^u$;
see \cite[Theorem~27.5]{Hu1975Linear-algebraic-g}.

We use $\Lie{g}$ to denote the Lie algebra of an algebraic group $G$.
Moreover, we denote by $\NNN_{\Lie{g}}$ the closed irreducible cone
of nilpotent elements inside $\Lie{g}$. Note that the exponential
$\exp \colon \Lie{g} \to G$ restricts to an isomorphism of affine
varieties $\exp \colon \NNN_{\Lie{g}} \to \UUU_{G}$.


\subsection{Principal bundles}

Our general reference for principal bundles is \cite{Se1958Espaces-fibres-}.
Again, in order to set up conventions, let us recall the basic terms.
Let $G$ be any algebraic group. A \emph{principal} $G$-bundle is a variety
$P$ with a right $G$-action
together with a $G$-invariant morphism
$\pi \colon P \to X$ such that locally on $X$, $\pi$ becomes a trivial principal
$G$-bundle after a finite \'etale base change. If one can choose
these \'etale base changes to be open injective immersions, then we say
$\pi$ is a \emph{locally trivial principal $G$-bundle}.

The most prominent example of a principal bundle in this article is the
following: let $G$ be an algebraic group and let $H$ be a closed subgroup.
Then $G \to G/H$ is a principal $H$-bundle; see~\cite[Proposition~3]{Se1958Espaces-fibres-}. If $H$ is a group without
characters, then the quotient $G/H$ is quasi-affine
(see \cite[Example~3.10]{Ti2011Homogeneous-spaces}) and if $H$ is
normal in $G$ or reductive, then $G/H$ is affine (see \cite[Theorem~3.8]{Ti2011Homogeneous-spaces}).
For any algebraic group $G$, any principal $G$-bundle over $\C$
is trivial; see Appendix~\ref{sec.PrincipalBundlesOverA1}.

\section{Construction of automorphisms of an algebraic group}
\label{sec.Auto}

In this section we introduce a construction of automorphisms
of algebraic groups that we use throughout this article.

Let $G$ be an algebraic group. Let $H \subseteq G$ be a closed subgroup
and let $\pi \colon G \to G/H$ be the quotient by left $H$-cosets.
For any morphism $f \colon G/H \to H$, the map
\[
	\varphi_f \colon G \longrightarrow G \, , \quad g \mapsto g f(\pi(g))
\]
is an automorphism of $G$ that preserves the quotient $\pi$.
Let $\rho \colon G \to H \backslash G$ be the quotient by right $H$-cosets.
Analogously to $\varphi_f$,
we define for any morphism $d \colon H \backslash G \to H$
the automorphism
\[
	\psi_d \colon G \longrightarrow G \, , \quad g \mapsto d(\rho(g)) g \, .
\]

We will frequently use
this construction in the following special situation. Assume that $H$ is a closed
unipotent subgroup, whence $G/H$ is quasi-affine. Let $X \subseteq G$ be
a closed curve that has only one smooth point at infinity, i.e.~there exists a projective curve $\bar{X}$ that contains $X$
as an open subset and $\bar{X} \setminus X$
consists only of one point that is a smooth point of $\bar{X}$.
Assume that the quotient $\pi$ restricts to
an embedding on $X$. Thus $X$ is a section of the
principal $H$-bundle $\pi^{-1}(\pi(X)) \to \pi(X)$.
Let $X'$ be another section of $\pi^{-1}(\pi(X)) \to \pi(X)$ and denote by
$s \colon \pi(X) \to X$ and $s' \colon \pi(X) \to X'$
the inverse maps of $\pi |_{X} \colon X \to \pi(X)$ and
$\pi |_{X'} \colon X' \to \pi(X)$, respectively.
Consider the morphism
\begin{equation}
	\label{eq.moveing}
	\pi(X) \longrightarrow H \, , \quad v \mapsto (s(v))^{-1} \cdot s'(v) \, .
\end{equation}
Since $G/H$ is quasi-affine and since $\pi(X)$ has only one smooth
point at infinity, the curve $\pi(X)$ is closed in any affine variety that contains
$G/H$ as an open subvariety. Since $H$ is unipotent and thus
an affine space, \eqref{eq.moveing}
can be extended to a morphism $f \colon G /H \to H$. Clearly, the automorphism
$\varphi_f$ satisfies $\varphi_f(X) = X'$. Roughly speaking, $\varphi_f$
moves $X$ into $X'$ along the fibers of $\pi$.

If $X$ happens to be the affine line,
then it is enough to assume that $G/H$ is quasi-affine in order to move
$X$ into another section along the fibers of $\pi$.
Indeed, since $G / H$ is quasi-affine,
there exists a retraction of $G / H$ to $\pi(X) \cong \AA$
and therefore the morphism in \eqref{eq.moveing} can be extended
to a morphism $f \colon G / H \to H$.
In summary we proved the following result and its analog for right coset spaces.

\begin{proposition}
	\label{prop.ConstructionOfAuto}
	Let $G$ be an algebraic group and let $H$ be a closed subgroup
	such that $G/H$ is quasi-affine. If $X$ is a closed curve in $G$ that is
	isomorphic to $\AA$ such that $\pi \colon G \to G/H$
	restricts to an embedding on $X$ and if $X'$ is another section of
	$\pi^{-1}(\pi(X)) \to \pi(X)$, then there exists an automorphism
	$\varphi$ of $G$ that preserves $\pi$ and maps $X$ onto $X'$.
\end{proposition}

\section{Embeddings of $\AA$ with unipotent image}\label{sec.embedwithunipotentimage}

The following result says that two embeddings $f_1$ and $f_2$ of $\AA$ into
an algebraic group $G$ are the same up to an automorphism
of $G$,
provided that $f_1(\AA)$ and $f_2(\AA)$ are unipotent subgroups of $G$.

\begin{proposition}
	\label{prop.firstExample}
	Let $G$ be any algebraic group and let $U$, $V$ be unipotent
	one-dimensional subgroups. For any isomorphism of varieties
	$\sigmaÊ\colon U \to V$, there exists an algebraic
	automorphism $\varphi$ of $G$ such that $\varphi |_U = \sigma$.
\end{proposition}

\begin{proof}
	If $G^u$ is one-dimensional, then $G^u = R_u(G)$ and
 	$G$ is isomorphic to $G^u \times G/G^u$
	as a variety; see Remark~\ref{rem.Triviality_of_affine_bundles}.
	In particular, $U = V = G^u$ and every
	automorphism of $U$ extends to $G$. Thus, we can assume that
	$G^u$ is at least two-dimensional and hence we can assume that
	$V \neq U$. This implies $V \cap U = \{ e \}$ and therefore multiplication
	$V \times U \to V U \subseteq G$ is an embedding.
	Hence, the quotient map $\pi \colon G \to G/U$ restricts to
	an embedding on $V$. Since $G / U$ is quasi-affine, the morphism
	\[
		\xymatrix{
			\pi(V) \ar[rr]^-{(\pi |_{V})^{-1}} && V
			\ar[r]^-{\sigma^{-1}} & U
		}
	\]
	extends to a morphism $f \colon G /U \to U$.  Hence
	the automorphism $\varphi_f$ of $G$ (see Section~\ref{sec.Auto})
	satisfies $\varphi_f(v) = v \cdot \sigma^{-1}(v)$
	for all $v \in V$. Using the quotient $\rho \colon G \to V \backslash G$
	one can similarly construct an automorphism $\psi_d$ of $G$ such that
	$\psi_d(u) = \sigma(u) \cdot u$ for all $u \in U$. It follows that
	$\varphi = \varphi_f^{-1} \circ \psi_d$ restricts to $\sigma$ on $U$.
\end{proof}

\section{Generic projection results}\label{sec.genproj}

The aim of this section is to prove results, which enable
us to quotient by unipotent subgroups such that the projection
restricts to a closed embedding or to a birational map on a given fixed curve.
These projection results will be applied in
Sections~\ref{sec.ReductionSemisimple} and Section~\ref{sec.ReductionSimple}
to reduced Theorem~\ref{thm:mainthm} to semisimple groups and simple groups,
respectively.
In Section~\ref{sec.Simple} we use these results in the heart of the proof of
Theorem~\ref{thm:mainthm}; namely for the case of embeddings into simple groups.

Let $V$ be a variety.
Throughout this 
paper we say that a property
is satisfied for generic $v \in V$ if there exists a dense open subset $O$ in $V$
such that the property is satisfied for all $v$ in $O$.

\subsection{Quotients that restrict to closed embeddings on a fixed curve}
Our first result in this section deals with arbitrary algebraic groups and quotients by
one-dimensional unipotent subgroups.

\begin{lemma}[Communicated by Winkelmann]
	\label{lem.OneDimProj}
	Let $G$ be an algebraic group and let $X \subseteq G$ be a closed curve
	that has only one smooth point at infinity.
	If the set of unipotent elements $\UUU_G$
	has dimension at least four, then, for a generic one-dimensional
	unipotent subgroup $U \subseteq G$, the quotient $G \to G /U$ restricts to
	a closed embedding on $X$.
\end{lemma}

\begin{remark}
	Using the exponential map
	$\exp: \NNN_{\Lie g} \to \UUU_G$,
	we consider the whole of one-dimensional
	unipotent subgroups of $G$ as the image of $\NNN_{\Lie g} \setminus \{Ê0 \}$
	under the quotient
	$\Lie g \setminus \{ 0 \} \to \PP(\Lie g)$. Note that
	this image is closed in $\PP(\Lie g)$ and therefore we can speak
	of a ``generic one-dimensional unipotent subgroup".
\end{remark}

\begin{proof}
	As already mentioned, the exponential
	restricts to an isomorphism of affine varieties
	$\exp \colon \NNN_{\Lie{g}} \to \UUU_G$.
	We denote by $F$ the set of all elements in $G$ of the form $y^{-1} x$
	with $x, y \in X$ and $x \neq y$. Let
	\[
		F' = \exp( \cone( \exp^{-1}(F \cap \UUU_G)) \subseteq
		\UUU_G \, ;
	\]
	where $\cone(M)$ denotes the union of
	all lines in $\NNN_{\Lie{g}}$ that pass through the origin
	and intersect $M$, for any subset $M$ of $\NNN_{\Lie{g}}$.
	Let $U \subseteq G$ be a one-dimensional unipotent subgroup.
	Thus $G \to G / U$ maps $X$ injectively onto its image if and only if
	$U \cap F' = \{ e \}$. However, $F'$ is a constructible
	subset of $\UUU_G$ of dimension at most three.

	Let $S \subseteq \Lie{g}$ be the union
	of all lines $D l_{x^{-1}}(T_x X)$, $x \in X$, where
	$l_g \colon G \to G$ denotes left multiplication by $g \in G$.
	Let $U \subseteq G$ be a one-dimensional unipotent subgroup.
	Thus $G \to G/U$ maps $X$ immersively onto its image if and only if
	$\mathfrak{u} \cap S \cap \NNN_{\Lie{g}} = \{Ê0 \}$ where $\mathfrak{u}$
	denotes the Lie algebra of $U$.
	Clearly, $S \cap \NNN_{\Lie{g}}$ is a constructible subset of $\NNN_{\Lie{g}}$
	of dimension at most two.
	
	Since $G / U$ is quasi-affine, the quotient $G \to G / U$ maps $X$
	properly onto its image, as long as the image is not a single point,
	since $X$ has only one smooth point at infinity.

	In summary, we proved that the restriction of $G \to G/U$ to $X$
	is injective, immersive and proper for
	a generic one-dimensional unipotent subgroup $U$ in $G$.
\end{proof}

\begin{remark}
	\label{rem.OneDimProj} The proof of Lemma~\ref{lem.OneDimProj}
	shows that we
	can replace $\UUU_G$ by some closed subset $W$ of $\UUU_G$
	that is a union of unipotent subgroups and has dimension at least four
	in order to prove that for a generic one-dimensional unipotent subgroup
	$U$ in $W$ the quotient $G \to G/U$ restricts to a closed embedding on $X$.
\end{remark}

Our second result deals with simple algebraic groups and quotients by
arbitrary unipotent subgroups.

\begin{proposition}
	\label{prop.UnipotProj}
	Let $G$ be a simple algebraic group of rank at least two and let
	$U \subseteq G$ be a unipotent subgroup.
	If $X \subseteq G$ is a closed smooth curve with only
	one smooth point at infinity,
	then there exists an automorphism $\varphi$ of $G$
	such that for generic $g \in G$ the projection $G \to G / gUg^{-1}$
	restricts to a closed embedding on $\varphi(X)$.
\end{proposition}

In order to prove this result, we have to show
that for generic $g \in G$ the projection $G \to G/ gUg^{-1}$
restricts to an injective and immersive map on $\varphi(X)$ for
a suitable automorphism $\varphi$.
If this is the case, then this restriction is automatically proper, since
$X$ has only one smooth point at infinity.

\begin{lemma}[Immersivity]	
	\label{lem.immersive}
	Let $G$ be a connected reductive algebraic group,
	$U \subseteq G$ a closed unipotent subgroup.
	If $X \subseteq G$ is a closed {irreducible}
	smooth curve such that $e \in X$
	and $T_e X$ contains non-nilpotent elements of the Lie algebra
	$\Lie{g}$, then for generic $g \in G$
	the projection $\pi_g \colon G \to G / g U g^{-1}$ restricts to an immersion
	on $X$.
\end{lemma}

\begin{proof}
	Denote by $\Lie{u}$ the Lie algebra of $U$.
	The kernel of the differential of $\pi_g$ in $e \in G$
	is the sub Lie algebra $\Ad(g) \Lie{u}$ of $\Lie{g}$,
	where $\Ad(g)$ denotes the linear isomorphism of $\Lie{g}$
	induced by the differential in $e$ of the automorphism of $G$ that is
	given by $h \mapsto g h g^{-1}$.
	Consider the morphism
	\begin{equation}
		\label{eq.immersive}
		G \times (\Lie{u} \setminus \{ 0 \})
		\to \PP(\Lie{g}) \, , \quad (g, v) \mapsto [\textrm{Ad}(g)v] \, ,
	\end{equation}
	where $[w]$ denotes the line through $0 \neq w \in \Lie{g}$.
	Since $G$ is not unipotent,
	the set of non-nilpotent elements is a dense open subset
	of $\Lie{g}$
	which maps via the projection $\Lie{g} \setminus \{ 0 \} \to \PP(\Lie{g})$
	to a dense open subset $O$. Since $\textrm{Ad}(g)v$ is nilpotent for all
	$v \in \Lie{u}$, the open set $O$ lies in the complement
	of the image of the morphism in  \eqref{eq.immersive}. Let
	\[
		S = \bigcup_{x \in X}  \PP(T_{e} (x^{-1} X)) \subseteq \PP(\Lie{g}) \, ,
	\]
	which is a locally closed irreducible curve in $\PP(\Lie{g})$.
	Hence, $\pi_g$ is immersive for $g \in G$ if and only
	if $S \cap \PP(\Ad(g) \Lie{u})$ is empty.
	By assumption $S \cap O$ is non-empty and thus
	there exists a finite subset $F$ of $S$ such that
	$S \setminus F \subseteq O$, since $S$ is irreducible. Thus
	$(S \setminus F) \cap \PP(\Ad(g) \Lie{u})$ is empty for all $g \in G$.
	We claim that
	\begin{equation}
		\label{eq:capAd=0}
		\bigcap_{g \in G} \Ad(g) \Lie{u} = \{ 0 \} \, .
	\end{equation}
	Using the isomorphism $\exp \colon \NNN_g \to \UUU_G$, \eqref{eq:capAd=0}
	is equivalent to 
	the intersection
	\begin{equation}
		\label{eq.intersection}
		\bigcap_{g \in G} gUg^{-1}
	\end{equation}	
	being trivial. Let $v$ be in the intersection in \eqref{eq.intersection}
	and let $N$ be the smallest closed subgroup of $G$
	that contains all conjugates $g v g^{-1}$ of $v$.
	Clearly, $N \subseteq U$.
	By \cite[Proposition~7.5]{Hu1975Linear-algebraic-g}, $N$ is
	connected and normal in $G$.
	Since the unipotent radical of $G$ is trivial, $N$ is trivial.
	Thus, $v = e$, which proves our claim. As a consequence of
	\eqref{eq:capAd=0},
	the intersection $F \cap \PP( \Ad(g) \Lie{u})$ is empty
	for generic $g \in G$. This proves the lemma.
\end{proof}

\begin{lemma}[Injectivity]
	\label{lem.injective}
	Let $G$ be a simple algebraic group of rank $\geq 2$ and let
	$U \subseteq G$ be a unipotent subgroup. If $X \subseteq G$
	is a closed irreducible
	curve such that $e \in X$ and $X$ contains non-unipotent
	elements, then for generic $g \in G$, the projection
	$\pi_g \colon G \to G / g U g^{-1}$ restricts to an injection on $X$.
\end{lemma}

\begin{proof}
	The strategy of the proof resembles the strategy of the proof
	of Lemma~\ref{lem.immersive}.
	Consider the morphism
	\[
		G \times U \to G \, , \quad (g, u) \mapsto gug^{-1} \, .
	\]
	Since $G$ is not unipotent, $G \setminus \UUU_G$ is dense and open
	in $G$,
	and it is contained in the complement of the image of the above morphism.
	Let us denote this open subset by $O$. Let
	\[
		S = \{ \, x^{-1} y \in G \ | \ x \neq y \in X  \, \} \, .
	\]
	Hence, $\pi_g$ is injective if and only if $S \cap gUg^{-1}$ is empty.
	By assumption $S \cap O$ is non-empty and thus there exists
	a curve (or finite set) $C \subseteq S$ consisting of unipotent elements
	such that $S \setminus C \subseteq O$ since $S$ is irreducible.
	Hence, $(S \setminus C) \cap gUg^{-1}$ is empty
	for all $g \in G$. Therefore it is enough to show that
	$C \cap gUg^{-1}$ is empty for generic $g \in G$.
	This can be achieved by showing that for all $e \neq v \in \UUU_G$
	the set
	\[
		F_v = \{ \, g \in G \ | \ v \in gUg^{-1} \,Ê\}
	\]
	has codimension $\geq 2$ in $G$. Indeed, if $\codim_G(F_v) \geq 2$
	for all $v \neq e$, then the dimension of
	\[
		F = \{ \, (v, g) \in C \times G \ | \ g \in F_v \, \}
	\]
	is less than the dimension of $G$. Hence, $F$ maps
	to a subset of codimension $\geq 1$ in $G$ via the natural projection
	$C \times G \to G$, which then implies that $C \cap gUg^{-1}$
	is empty for generic $g \in G$.
	
	So let us prove that $\codim_G F_v \geq 2$. Denote by
	$\Cl_G(v)$ the conjugacy class of $v$ in $G$.
	By using the orbit map
	$G \to \Cl_G(v)$, $g \mapsto g^{-1}v g$
	one can see that $\codim_G F_v$ is the same as the
	codimension of $U \cap \Cl_G(v)$ in $\Cl_G(v)$.
	Since $G$ is semisimple,
	by~\cite[Proposition~6.7]{Hu1995Conjugacy-classes-} we have
	\[
		\dim U \cap \Cl_G(v) \leq \frac{1}{2} \dim \Cl_G(v) \, .
	\]
	Hence, it remains to show that $\Cl_G(v)$
	has dimension $\geq 3$, since
	the dimension of $\Cl_G(v)$ is even by
	\cite[Proposition~6.7]{Hu1995Conjugacy-classes-}.
	This is in fact equivalent to the statement that
	the centralizer $\Cent_G(v)$
	having codimension $\geq 3$ in $G$. The latter is true
	by the following argument. The unipotent radical $R_u(\Cent_G(v))$
	is not trivial since the one-dimensional unipotent group
	which contains $v \neq e$ is normal in $\Cent_G(v)$.
	Clearly, $\Cent_G(v)$ lies inside
	the normalizer $N_G(R_u(\Cent_G(v))$. However, this normalizer is contained
	in some parabolic subgroup $P$ that itself is the normalizer of some
	non-trivial unipotent subgroup of $G$;
	see~\cite[Corollary 30.3A]{Hu1975Linear-algebraic-g}.
	Since $G$ is reductive,
	this implies that $P$ is a proper subgroup of $G$.
	Since $G / \Cent_G(v) \to G$, $g \mapsto g^{-1}v g$
	is injective, $G$ is an affine variety, and $G/P$ is
	projective and of positive dimension,
	it follows that $\Cent_G(v)$ must be a proper
	subgroup of $P$. Since $P$ is connected, we have
	$\dim \Cent_G(v) < \dim P$. Since $G$ is simple
	and since the rank of $G$ is at least two, it follows from
	Lemma~\ref{lem.enoughUnipotentsInAParabolic} that
	$\dim R_u(P^-) \geq 2$. Here $P^-$ is the opposite
	parabolic subgroup to $P$ with respect to some maximal torus
	that is contained in some Borel subgroup which in turn is contained in $P$;
	see Appendix~\ref{sec.OppositeParabolicSubgroup}. This implies
	that the codimension of $P$ in $G$ is at least $2$ by
	Lemma~\ref{lem.dimG_is_dimR_uPminus_plus_dimP}. This in turn implies
	that $\Cent_G(v)$ has codimension $\geq 3$ in $G$, which proves the lemma.
\end{proof}

\begin{proof}[Proof of Proposition~\ref{prop.UnipotProj}]
	Since $G$ is simple, it is a so called flexible variety; see
	\cite[\S0]{ArFlKa2013Flexible-varieties}. Hence,
	there exists an automorphism
	$\varphi$ of $G$ such that $\varphi(X)$
	contains non-unipotent elements, $e \in \varphi(X)$ and the tangent space
	$T_e X$ contains non-nilpotent elements of the Lie algebra $\Lie{g}$;
	see \cite[Theorem 4.14, Remark 4.16 and Theorem~0.1]
	{ArFlKa2013Flexible-varieties}.
	By Lemma~\ref{lem.immersive} and Lemma~\ref{lem.injective}, for
	generic $g \in G$ the projection $\pi_g \colon G \to G / gUg^{-1}$
	restricted to $\varphi(X)$ is immersive and injective.
	As already mentioned, if this is the case, then $\pi_g |_{\varphi(X)}$
	is proper. This finishes the proof.
\end{proof}

\subsection{Quotients that restrict
to birational maps on a fixed curve}
Let us introduce the following notation. If $G$ is an algebraic
group, then for any $u \in \UUU_G \setminus \{ e \}$ we denote by ${\C^+}(u)$ the
one-dimensional unipotent subgroup of $G$ that contains $u$.
Roughly speaking the next lemma says: Under certain assumptions,
a curve $C$ in an affine homogeneous $G$-variety $Y$
projects birationally onto its image if we quotient $Y$ by ${\C^+}(u)$
where $u$ belongs to a dense subset of $\UUU_G$.

\begin{lemma}
	\label{lem.gen-projection-birational}
	Let $Y$ be an affine homogeneous $G$-variety
	where $G$ is a connected algebraic group acting from the right.
	We assume that generic elements in $\UUU_G$
	act without fixed point on $Y$. Moreover, we assume that
	for all $y$ in $Y$, every fiber of
	the morphism
	\[
		\rho_y \colon \UUU_G \to Y \, , \quad u \mapsto yu
	\]
	has codimension at least three in $\UUU_G$.
	If $C \subseteq Y$ is a closed curve,
	then there exists a dense subset in $\UUU_G$
	consisting of elements $u$ such that ${\C^+}(u)$
	acts without fixed point on $Y$ and the algebraic quotient
	$S_u \to S_u \aquot {\C^+}(u)$ restricts to a birational
	morphism on $C$, where $S_u$ denotes the smallest closed
	affine surface in $Y$ that contains all ${\C^+}(u)$-orbits passing
	through $C$.
\end{lemma}

\begin{remark}
	The algebraic quotient $S_u \aquot {\C^+}(u)$
	is the spectrum of the ring of functions on $S_u$ that are invariant under
	the action of ${\C^+}(u)$. In fact, $S_u \aquot {\C^+}(u)$
	is an irreducible affine curve,
	see \cite[Theorem~11.7]{Ma1986Commutative-ring-t} and
	\cite[Corollary 1.2, Theorem~3.2]{OnYo1982On-Noetherian-subr}.
\end{remark}

\begin{proof}[Proof of Lemma~\ref{lem.gen-projection-birational}]
	Let $c_0 \in C$ and let $K_{c_0}$ be the union of the orbits
	$c_0 {\C^+}(u)$, $u \in \UUU_G \setminus \{e\}$ where
	$c_0 {\C^+}(u)$ is either equal to $\{ c_0 \}$ or it
	contains points of $C$ different from $c_0$. In other words,
	\[
		K_{c_0} =
		\bigcup_{e \neq u \in \UUU_G \, \textrm{such that} \, c_0 u \in C}
		c_0 {\C^+}(u)  \, .
	\]
	With the aid of the exponential map
	$\exp \colon \NNN_{\Lie{g}} \to \UUU_G$ we define
	\[
		N_{c_0} =
		\bigcup_{e \neq u \in \rho_{c_0}^{-1}(C)} {\C^+}(u)
		= \exp(\textrm{cone}(\exp^{-1}(\rho_{c_0}^{-1}(C)))) \subseteq \UUU_G \, .
	\]
	One can see that $N_{c_0} = \rho_{c_0}^{-1}(K_{c_0})$. In particular, we have
	for $u \in \UUU_G \setminus N_{c_0}$ that
	$c_0 {\C^+}(u)$ intersects $C$ only in the point $c_0$.
	Since all the fibers of $\rho_{c_0}$ have codimension at least three in
	$\UUU_G$ and since $\dim C = 1$, it follows that
	$\dim \rho_{c_0}^{-1}(C) \leq \dim \UUU_G - 2$.
	By the construction of $N_{c_0}$ we get now
	\[
		\dim N_{c_0} \leq \dim \UUU_G - 1 \, .
	\]
	
	Take a countably infinite subset $C_0 \subseteq C$. Since our ground field
	is uncountable, the intersection
	$\bigcap_{c_0 \in C_0} \UUU_G \setminus N_{c_0}$ is
	dense in $\UUU_G$.
	Let $u \in \UUU_G$ be an element that acts without fixed point on
	$Y$ and such that
	$u \not\in \bigcup_{c_0 \in C_0} N_{c_0}$.
	Since a fiber of $S_u \to S_u \aquot {\C^+}(u)$ over a generic point
	of $S_u \aquot {\C^+}(u)$ is a ${\C^+}(u)$-orbit,
	it follows that {infinitely many fibers of $C \to S_u \to S_u \aquot \C^+(u)$
	consist only of one point.} Thus,
	$C$ is mapped birationally onto the algebraic quotient.
\end{proof}

\begin{remark}
	\label{rem.gen-projection-birational}
	The proof of the Lemma~\ref{lem.gen-projection-birational}
	shows the following: If there exist infinitely many
	$c_0$ in $C$ such that $\rho_{c_0}^{-1}(C) \leq \dim \UUU_G - 2$,
	then the statement of the lemma holds. In particular, the
	statement of the lemma holds, if there are infinitely many
	$c_0 \in C$ such that all fibers of
	$\rho_{c_0} \colon \UUU_G \to Y$ have codimension at least
	two in $\UUU_G$ and $c_0 \UUU_G \cap C$ is finite.
\end{remark}


\begin{corollary}
	\label{cor.Birational_projection}
	Let $G$ be a connected algebraic group such that
	$\dim G \geq 3$, $\dim \UUU_G \geq 2$ and $G = G^u$. If $C \subseteq G$
	is a closed irreducible curve, then there exists an automorphism
	$\varphi$ of $G$ and a dense subset of $\UUU_G$ consisting
	of elements $u$
	such that $G \to G / {\C^+}(u)$ maps $\varphi(C)$ birationally
	onto its image.
\end{corollary}

\begin{proof}
	If $G$ is a unipotent group, the statement is clear, since $\dim G \geq 3$.
	Thus we can assume that $\UUU_G$ is a proper
	subset of $G$.
	Since $G = G^u$, the variety $G$ is flexible.
	Fix some point $c_0$ in $C$. By
	\cite[Theorem~0.1]{ArFlKa2013Flexible-varieties}
	there exists an automorphism $\varphi$
	of $G$ that fixes $c_0$ and the image $\varphi(C)$ intersects $c_0 \UUU_G$
	only in finitely many points. Thus we can assume that $c_0 \UUU_G \cap C$
	is finite. The fiber over $c \in C$ of the morphism
	\begin{equation}
		\label{eq.morphism}
		\{ \, (c, u) \in C \times \UUU_G \ | \ cu \in C \, \} \to C \, , \quad
		(c , u) \mapsto c
	\end{equation}
	is isomorphic to $c \UUU_G \cap C$.
	Since $C$ is irreducible, the subset of $C$
	given by
	\[
		C' := \{ \, c \in C \ | \ C \subseteq c \UUU_G \, \}
	\]
	consists of exactly those points for which the fiber of
	\eqref{eq.morphism} is not finite.
	Note that $C'$ is closed in $C$.
	Since $c_0 \UUU_G \cap C$ is finite,
	$C'$ is a proper subset of $C$. Since $C$ is irreducible, it follows now that
	the generic fiber
	of \eqref{eq.morphism} is finite, {i.e.~$c \UUU_G \cap C$ is finite for
	generic $c$ in $C$}.
	Since $\dim \UUU_G \geq 2$, it follows
	that for all $c \in C$
	the fibers of the map $\rho_c \colon \UUU_G \to G$, $\rho_c(u) = cu$
	have codimension at least two in $\UUU_G$.
	The corollary follows from
	Remark~\ref{rem.gen-projection-birational} applied
	to the homogeneous $G$-variety $Y = G$.
\end{proof}

\section{Reduction to semisimple groups}
\label{sec.ReductionSemisimple}

In this section we reduce the proof of Theorem~\ref{thm:mainthm}
to semisimple groups.

\begin{lemma}
	\label{lem.reduction_to_characterless}
	Let $G$ be a connected algebraic group with $G = G^u$
	and let $X$ be an affine variety that admits no non-constant invertible
	function $X \to {\C^*}$.
	{Moreover, let $n$ be a non-negative integer.}
	Then,
	all closed embeddings of $X$ into $G \times (\C^*)^n$ are
	equivalent if and only if all closed embeddings of $X$
	into $G$ are equivalent.
\end{lemma}

\begin{proof}
	Let $f_i \colon X \to G \times (\C^*)^n$, $i=1, 2$
	be two closed embeddings.
	By assumption, $f_i(X)$ lies in some fiber of
	$\pi \colon G \times (\C^*)^n \to (\C^*)^n$ for $i = 1, 2$.
	After multiplying with a suitable element of $G \times (\C^*)^n$
	we can assume that
	$f_1(X)$ and $f_2(X)$ lie in the same fiber of $\pi$.
	Since any automorphism of
	one fiber can be extended to $G \times (\C^*)^n$, this proves
	the if-part of the proposition.
	
	The other direction works much the same way by using the fact, that every
	automorphism of $G \times (\C^*)^n$ permutes the fibers of $\pi$, since
	$G = G^u$ and thus there are no
	non-constant invertible functions
	$G \to \C^\ast$;
	see \cite[Theorem~3]{Ro1961Toroidal-algebraic}.
\end{proof}

\begin{lemma}
	\label{lem.Product}
	Let $G$ be a connected algebraic group. Then
	$G$ is isomorphic as a variety to $G^u \times (\C^*)^n$
	for a certain non-negative integer $n$.
\end{lemma}

\begin{proof}
	Note that $G/G^u$ is a torus, since it is connected and
	contains only semisimple elements; see
	\cite[Proposition~21.4B and Theorem~19.3]{Hu1978Introduction-to-Li}.
	Let $T$ be a maximal torus of $G$.
	Since $G^u$ is normal in $G$,
	and since $G^u$ and $T$ generate $G$
	(see \cite[Theorem~27.3]{Hu1975Linear-algebraic-g}) we have
	$TG^u= G$.
	In particular, $T$ is mapped surjectively onto the torus $G/G^u$ via
	the canonical projection $\pi \colon G \to G/G^u$.
	Thus we get a short exact sequence
	\[
		1 \longrightarrow G^u \cap T \longrightarrow
		T \stackrel{\pi |_T}{\longrightarrow} G/G^u
		\longrightarrow 1 \, .
	\]
	By Lemma~\ref{lem.Intersetion_torus_normal-subgroup},
	$G^u \cap T$ is a torus.
Thus the above short exact sequence splits{;
	see~\cite[\S16.2]{Hu1975Linear-algebraic-g}}.
	In particular, the associated section yields a trivialization of
	 $\pi \colon G \to G/ G^u$ as a principal $G^u$-bundle,
	which proves the lemma.
\end{proof}

\begin{lemma}
	\label{lem.Intersetion_torus_normal-subgroup}
	Let $G$ be any connected algebraic group and let $H$ be a closed
	connected normal subgroup of $G$. If $T$ is a maximal torus of $G$,
	then $T \cap H$ is a maximal torus of $H$.
\end{lemma}

\begin{proof}
	Let $T' \subseteq H$ be a maximal torus that contains the connected
	component of the identity element
	$(T \cap H)^{\circ}$ which is also a torus.
	Since all maximal tori in $G$
	are conjugate, there exists $g \in G$ such that $g^{-1}T' g \subseteq T$.
	By the normality of $H$ we get
	$g^{-1}T' g \subseteq (T \cap H)^{\circ}$. Hence
	\[
		g^{-1} (T \cap H)^{\circ} g \subseteq g^{-1}T' g \subseteq
		(T \cap H)^{\circ} \, .
	\]
	Thus $(T \cap H)^{\circ} = g^{-1} T' g$ is a maximal torus of $H$
	(note that all maximal tori of $H$ are conjugate, since $H$
	is connected).
	Now, if there exists $x \in T \cap H \setminus (T \cap H)^{\circ}$,
	then clearly $x$ is semisimple and centralizes the
	torus $(T \cap H)^{\circ}$. However,
	this implies that $\{ x \} \cup (T \cap H)^{\circ}$ lies in a torus of $H$,
	since $H$ is connceted
	(see \cite[Corollary~B, \S22.3]{Hu1975Linear-algebraic-g}).
	This contradicts the maximality of $(T \cap H)^{\circ}$ and thus
	$T \cap H = (T \cap H)^{\circ}$ is a maximal torus of $H$.	
\end{proof}

We are now in position, to formulate our main result of this section.
%
%
\begin{theorem}
	\label{thm.reduction_to_semisimple}
	Let $G$ be a connected algebraic group with
	$G = G^u$. If $G$ is not semisimple
	and not isomorphic to $\AA^3$ as a variety, then
	all embeddings of $\AA$ into $G$ are equivalent.
\end{theorem}

Using Lemma~\ref{lem.reduction_to_characterless} and
Lemma~\ref{lem.Product},
Theorem~\ref{thm.reduction_to_semisimple} reduces the proof of Theorem~\ref{thm:mainthm} to the case
that the group under consideration is semisimple
and not isomorphic to $\SL_2$ or $\PSL_2$; compare with the
proof of Theorem~\ref{thm:mainthm} in Section~\ref{sec.Simple}.

The rest of this section is devoted to the proof of
Theorem~\ref{thm.reduction_to_semisimple}. First we have to do some
preliminary work.

\begin{proposition}
	\label{prop.Product}
	Let $K$ be a connected group that contains
	non-trivial unipotent elements
	and let $H$ be a semisimple group
	{(which is non-trivial by convention)}. Then
	all embeddings of $\AA$ into $K \times H$ are
 equivalent.
\end{proposition}

\begin{proof}
	Let $\AA \cong X \subseteq K \times H$ be an embedding.
	We can assume that the canonical projection
	$\pi_H \colon K \times H \to H$ maps $X$ birationally onto its image;
	compare Lemma~\ref{lemma:biratproj} below.
	We can apply Corollary~\ref{cor.Birational_projection} to the group $H$
	and the curve $\pi_H(X)$,
	since $H$ is a {(non-trivial)} semisimple group.
	Hence we can assume that there exists
	a one-dimensional unipotent subgroup $U \subseteq H$
	such that the composition
	\[
		\rho \colon K \times H \stackrel{\pi_H}{\longrightarrow}
		H \stackrel{}{\longrightarrow} H/U
	\]
	restricts to a birational morphism $X \to \rho(X)$.
	Let $E$ be the finite subset of elements $z$ in $H/U$ such that
	the fiber over $z$ of $\rho |_X$ contains more than one element.
	Moreover, let $X'$ be the finite subset of $X$
	of critical points of $\rho |_X$.
	For a morphism $f \colon K \to U$ consider
	the two properties:
	\begin{enumerate}[$i)$]
	\item \label{enum.Function}
	For every $z \in E$ and for every pair $(k, h),  (k', h')$ in
	$\rho^{-1}(z) \cap X$ with $(k, h) \neq (k, h')$ we have
	\[
		h f(k) \neq h' f(k') \, .
	\]
	\item \label{enum.Differential}
	For every $x' \in X'$ the differential of
	\[
		\eta_f \colon X \longrightarrow H \, , \quad x \mapsto \pi_H(x) f(\pi_K(x))
	\]
	in $x'$ is non-vanishing.
	\end{enumerate}
	If we consider $U$ as a one-dimensional vector space,
	then for every pair of points $(k, h) \neq (k', h')$ in $\rho^{-1}(z) \cap X$,
	the expression $h f(k) = h' f(k')$
	defines a non-trivial affine linear equation for $f$
	in the vector space of maps $K \to U$
	(note that by assumption $(h')^{-1}h$ lies in $U$).
	Moreover, we claim that for every $x' \in X'$ the vanishing
	of the differential $D_{x'} \eta_f$
	defines a non-trivial affine linear
	equation for $f$ in the vector space of maps $K \to U$.
	Indeed, let $x' \in X'$ and let $W$ be an open neighbourhood of $\rho(x')$
	in $H/U$ in the Euclidean topology such that $H \to H/U$
	gets trivial over $W$. Then the map $\eta_f$ can be written
	in a Euclidean neighbourhood $U_{x'}$ in $X$ around $x'$ as
	\[
		U_{x'} \stackrel{}{\longrightarrow} W \times U \, , \quad
		x \mapsto (\rho(x), q(x)f(\pi_K(x))) \, ,
	\]
	where $q \colon U_{x'} \to U$ defines a {holomorphic}
	map (that does not depend on $f$) with the following property:
	If the differential $D_{x'} q$ vanishes, then the differential
	$D_{x'}(\pi_K |_X)$ is non-vanishing. Now the vanishing of the
	differential of $\eta_f$ in $x'$ is equivalent to the vanishing
	of the linear map
	\[
		D_{x'} q + D_{\pi_K(x')} f \circ D_{x'}( \pi_K |_X) \colon T_{x'} X \to U \, ,
	\]
	where we consider again $U$ as a one-dimensional vector space.
	However, this last condition defines a non-trivial
	affine linear equation for $f$. This proves the claim.
	In summary, we showed that there exists $f_0 \colon K \to U$ such that
	$\ref{enum.Function})$
	and $\ref{enum.Differential})$ are satisfied.
	Define
	\[
		\psi_0 \colon K \times H \to K \times H \, , \quad
		(k, h) \mapsto (k, h f_0(k)) \, .
	\]
	Then, the restriction of $\pi_H$ onto $\psi_0(X)$ is injective and immersive,
	since $f_0$ satisfies~$\ref{enum.Function})$ and $\ref{enum.Differential})$.
	Since $X \cong \AA$, the map $\pi_H$ restricts to an embedding
	on $\psi_0(X)$. Hence, after composing $\psi_0$
	with an automorphism of $K \times H$
	we can assume that $X$ lies in $H$;
	{see Proposition~\ref{prop.ConstructionOfAuto}}. Let $V \subseteq K$
	be any one-dimensional unipotent subgroup and let $f_1 \colon H \to V$
	be a morphism that restricts to an isomorphism on $X$.
	Let $\psi_1$ be defined as
	\[
		\psi_1 \colon K \times H \to K \times H \, , \quad
		(k, h) \mapsto (k f_1(h), h) \, .
	\]
	It follows that $\pi_K$ maps $\psi_1(X)$ isomorphically onto $V$.
	Hence, there exists an automorphism of
	$K \times H$ that sends $\psi_1(X)$ into $V$;
	{see Proposition~\ref{prop.ConstructionOfAuto}}.
	Thus the proposition follows from Proposition~\ref{prop.firstExample}.
\end{proof}

\begin{lemma}\label{lemma:biratproj}
Let $H$ be an algebraic group with $\dim H^u\geq 2$ and let $K$ be
any affine variety.
For any closed curve $X\subset K\times H$ that is isomorphic to $\AA$,
there exists an automorphism $\psi$ of $K\times H$ such that the canonical
projection $\pi_H \colon K \times H \to H$ restricts to a birational map
$X \to \psi(X)$.
\end{lemma}

\begin{proof}
We only consider the case that $K$ has dimension at least 1 (otherwise $\pi_H$ restricts to an embedding on $X$). By the same argument, we can
assume that the canonical projection $\pi_K \colon K\times H \to K$ is non-constant
on $X$.
We will use automorphisms of the form
\begin{equation}
	\label{eq:auto}
	\psi_f \colon K\times H\to K\times H \, , \quad (k,h)\mapsto (k,f(k)h) \, ,
\end{equation}
where $f\colon K\to U$ is a map to a one-dimensional unipotent subgroup $U$ of $H$. 

Let us first consider the case where $\pi_H(X)$ is zero-dimensional, i.e.~$\pi_H(X)$ is a point, and show that we can change that by applying an automorphism of the form~\eqref{eq:auto}.
Without loss of generality, we may assume that the point $\pi_H(X)$ is the
identity element $e$ of $H$, i.e.~$X$ lies in $K$.
Choose any non-trivial one-dimensional unipotent subgroup $U \subseteq H$
and let $f \colon K \to U$ be a morphism that is non-constant on $X$.
Thus $\pi_H(\psi_f(X))$ is one-dimensional.


By the above we may assume that $\pi_H(X)$ is one-dimensional. We consider a regular value {$h \in \pi_H(X)$} of the map $\pi_{H} |_X \colon X \to \pi_H(X)$ in the smooth locus of $\pi_H(X)$. {Since $\pi_K |_X$ is non-constant,
we can assume that the differential of $\pi_K |_X$ is non-vanishing
in every point of the fiber $(\pi_H |_X)^{-1}(h)$.}
As before we may assume that $h$
is the identity element $e$ of $H$. Denote by
\[
	x_1 = (k_1,e) \, , \ldots \, , x_n=(k_n,e)
\]
the elements of the fiber $(\pi_{H}|_X)^{-1}(e)$.
Note that for $i = 1, \ldots, n$ the lines $D_{x_i}\pi_H(T_{x_i}X)$
are all the same in $T_eH$ (otherwise
$e$ lies not in the smooth part of $\pi_H(X)$). Let us denote this line
in $T_eH$ by $l$.
We next establish that there is a automorphism $\psi_f$ of the form~\eqref{eq:auto} such that for all $1\leq i<j\leq n$
\begin{itemize}
\item $\psi_f(x_i)=x_i$,
\item $\psi_f(X)\cap \pi_H^{-1}(e)=\{x_1,\ldots,x_n\}$, and
\item
$D_{x_i}\pi_H(T_{x_i}\psi_f(X)) \neq D_{x_j}\pi_H(T_{x_j}\psi_f(X))$.
\end{itemize}
Since $\dim H^u \geq 2$, we find a
one-dimensional unipotent subgroup $U\subset H$ such that
$T_eU$ differs from $l$ and such that $\pi_H(X)\cap U$ is finite.
The first two conditions are arranged by choosing
an $f \colon K \to U$ with
\begin{equation}
	\label{eq:2propoff}
	f(k)=e \quad \text{ for all }
	k\in \pi_K\left(\{x_1,\ldots,x_n\}\cup\pi_H^{-1}(\pi_H(X)\cap U)\right).
\end{equation}
Let $t_i=v_i\oplus w_i\in T_{x_i}X\subset T_{k_i}K\oplus T_eH$ be non-zero tangent vectors to $X$ at $x_i$ for all $1\leq i\leq n$. We calculate
$D_{x_i}(\pi_H\circ\psi_f)(t_i)$ for any $f$
satisfying~\eqref{eq:2propoff}. In fact, by writing
$T_{(k_i,e)}(K\times H)=T_{k_i}K\oplus T_eH$, we get that
\[
	D_{(k_i, e)}\psi_f
	=\begin{pmatrix}
		\id & 0\\
		D_{k_i} f & \id
	\end{pmatrix} \, ,
\]
and thus
\[
	D_{x_i}(\pi_H\circ\psi_f)(t_i)=D_{k_i}f(v_i)+w_i \, .
\]
Since $D_{k_i}f(v_i)\in T_eU$, {$v_i \neq 0$}, $0 \neq w_i\in l$
and $l \neq T_eU$,
we see that we may choose $f$ (by prescribing its derivative at
$k_i$ for all $1\leq i\leq n$) such that
\[
	D_{x_i}(\pi_H\circ\psi_f)(t_i)\et D_{x_i}(\pi_H\circ\psi_f)(t_j)
\]
are linearly independent for all $1\leq i<j\leq n$.
Let $Y = \psi_f(X)$.

We conclude the proof by observing that $Y \to \pi_H(Y)$ is birational. Indeed,
let $Z = \pi_H(Y)$ and let $\eta \colon \tilde{Z} \to Z$ be the normalization,
which is birational. As $Y$ is smooth, $Y \to Z$ factorizes as $Y \to \tilde{Z} \to Z$.
Since $\eta$ factorizes through the blow-up of $Z$ in $e$,
$Z$ is closed in $H$, and the tangent directions of the branches of
$Z$ in $e$ are all different, it follows that $\eta^{-1}(e)$ consists of $n$ points,
say $v_1, \ldots, v_n$. After reordering the $v_1, \ldots, v_n$, we can assume
that $Y \to \tilde{Z}$ maps $x_i$ to $v_i$ for all $i$.
Since $Y \to \pi_H(Y)$ is immersive in $x_i$, it follows that $Y \to \tilde{Z}$
is \'etale in $x_i$ for all $i$.
Thus the fiber of $Y \to \tilde{Z}$ over $v_i$ consists only of $x_i$ and it is reduced
for all $i$. Since $Y \cong \C$, it follows that $\tilde{Z} \cong \C$ and therefore
$Y \to \tilde{Z}$ is an isomorphism. This proves that
$Y \to \tilde{Z} \to Z$ is birational.
\end{proof}


\begin{proof}[Proof of Theorem~\ref{thm.reduction_to_semisimple}]
	Note that if $F$ is a connected reductive group,
	then $F^u$ is semisimple or trivial. Indeed, by
	\cite[Proposition~14.2]{Bo1991Linear-algebraic-g} the derived
	group $[F, F]$ is semisimple (or trivial) and in fact, $F^u = [F, F]$,
	since $[F, F]$ contains all root subgroups with respect to any maximal torus
	{of $F$}.
	
	By definition, the quotient group $G / R_u(G)$
	is connected and reductive and since $G = G^u$, we get
	$G / R_u(G) = (G/R_u(G))^u$. Thus $G / R_u(G)$ is semisimple or trivial by
	the proceeding paragraph. By Remark~\ref{rem.Triviality_of_affine_bundles},
	$G$ is isomorphic as a variety to the product of
	$R_u(G)$ and $G / R_u(G)$. Now, we distinguish two cases:
	\begin{enumerate}[i)]
	\item $G \neq R_u(G)$.
	Since $G$ is not semisimple by assumption, the radical $R_u(G)$ is not trivial.
	Thus we can apply Proposition~\ref{prop.Product} to the non-trivial groups
	$K = R_u(G)$ and $H = G / R_u(G)$ and get the result.
	\item $G = R_u(G)$. Thus $G$ is isomorphic as a variety to $\CC^n$
	where $n$ is a non-negative integer $\neq 3$. Clearly, we can assume
	that $n > 1$. If $n = 2$, then the result follows from
	 the Abhyankar-Moh-Suzuki Theorem
	\cite[Theorem~1.2]{AbMo1975Embeddings-of-the-},
	\cite{Su1974Proprietes-topolog} and if $n \geq 4$, then the result follows
	from Jelonek's Theorem~\cite[Theorem~1.1]{Je1987The-extension-of-r}.
	\end{enumerate}
\end{proof}

\section{Reduction to simple groups}
\label{sec.ReductionSimple}

The aim of this section is to reduce our problem to the case of a simple
algebraic group.

\begin{proposition}
	\label{prop.non-semisimple}
	Let $G$ be a semisimple algebraic group that is not simple.
	Then, two embeddings of the affine line into $G$ are the same
	up to an automorphism of $G$.
\end{proposition}

For the proof we need three lemmata, which we also use later on.

\begin{lemma}
	\label{lem.twosubgroups}
	Let $G$ be a connected algebraic group
	and let $K$, $H$ be closed {connected} subgroups such that
	$KH$ is closed in $G$ and $K \backslash G$ is quasi-affine.
	If $X \subseteq K H$ is a closed curve that
	is isomorphic to $\AA$
	and if the canonical projection $G \to K \backslash G$ restricts to
	an embedding on $X$,
	then there exists an automorphism $\psi$ of $G$ with
	 $\psi(X) \subseteq H$.
\end{lemma}

\begin{proof}
	Let $K \times^{K \cap H} H \to K / K \cap H$ be the bundle
	associated to the principal $K \cap H$-bundle $K \to K / K \cap H$
	with fiber $H$; compare Appendix~\ref{sec.PrincipalBundlesOverA1}.
	The natural morphism
	$K \times^{K \cap H} H \to KH$ is bijective and
	since $KH$ is a smooth irreducible variety (note that $KH$ is closed in $G$),
	it follows from Zariski's Main Theorem
	\cite[Corollaire~4.4.9]{Gr1961Elements-de-geomet-III_1}
	that $K \times^{K \cap H} H \to KH$ is an isomorphism.
	Thus, multiplication $m \colon K \times H \to KH$
	is a principal $K \cap H$-bundle;
	see \cite[Proposition~4]{Se1958Espaces-fibres-}.
	
	Since $\AA \cong X \subseteq K H$, there exists a section
	$Y \subseteq K \times H$ over $X$ by Theorem~\ref{thm.trivial_over_A1}.	
	Denote by $\pr_K \colon K \times H \to K$ the canonical projection to $K$
	and by $\rho \colon G \to K \backslash G$ the quotient morphism.
	By assumption, $\rho \circ m |_Y \colon Y \to \rho(X)$
	is an isomorphism. Since $\rho(X) \cong \AA$ and since $K \backslash G$
	is quasi-affine,
	\[
		\rho(X) \to K \, , \quad v \mapsto
		\left(\pr_K \circ (\rho \circ m |_Y)^{-1}(v) \right)^{-1}
	\]
	extends to a morphism $d \colon K \backslash G \to K$. Let
	$\psi_{d}$ be
	the automorphism of $G$ constructed in Section~\ref{sec.Auto}.
	One can easily see that $\psi_{d}(X) \subseteq H$.
\end{proof}

	\begin{lemma}
		\label{lem.subgroup}
		Let $G$ be an algebraic group with $G = G^u$ and
		let $K$ be a closed proper subgroup of $G$.
		Assume that $\UUU_G$ has dimension
		at least four. If $X \subseteq K$ is a closed curve that is
		isomorphic to $\AA$, then there exists an automorphism
		$\varphi$ of $G$ such
		that $\varphi(X)$ is a unipotent subgroup of $G$.
	\end{lemma}
	
	\begin{proof}
	Note that the connected components of $K/K^u$ are tori.
	Since $X$ is the affine line, it lies in some fiber of $K \to K/ K^u$.
	Hence, after multiplying from the left with a suitable element of $K$,
	we can assume that $X \subseteq K^u$.
	Since $K^u$ does not contain all unipotent elements of $G$
	(otherwise $K = G$, since $G = G^u$),
	by Lemma~\ref{lem.OneDimProj} there exists a one-dimensional unipotent
	subgroup $U \subseteq G$ such that $U \cap K^u = \{e \}$
	and $\pi \colon G \to G/U$ induces an embedding on $X$.
	
	Choose an isomorphism $\pi(X) \cong U$
	and let $f \colon G/U \to U$ be an extension of it.
	The automorphism $\varphi_f$ of $G$ (see Section~\ref{sec.Auto})
	leaves $K^u U$ invariant.
	Since $U \cap K^u = \{Êe \}$, there is a canonical projection
	$K^u U \to U$. Since $X \subseteq K^u$,
	the composition
	\[
		X \stackrel{\varphi_f}{\longrightarrow} \varphi_f(X)
		\subseteq K^u U \longrightarrow U
	\]
	is an isomorphism. In particular, we can assume that $X \subseteq K^u U$
	and that $\rho \colon G \to K^u \backslash G$ induces an
	embedding on $X$.
	Now, we can apply Lemma~\ref{lem.twosubgroups} to the group $G$
	and the closed {connected} subgroups
	$K^u$ and $U$ to get an automorphism
	$\varphi$ of $G$ such that $\varphi(X) = U$.
%
\end{proof}

\begin{lemma}
	\label{lem.almostproducts}
	Let $K$, $H$ be non-trivial connected algebraic groups
	with $K = K^u$, $H = H^u$ and let $Z \subseteq K \times H$
	be a finite central subgroup. Assume that
	$\dim \mathcal{U}_H \geq 4$. If $X \subseteq (K \times H) / Z$
	is a closed curve that is isomorphic to $\AA$,
	then there exists an automorphism $\varphi$ of $(K \times H) / Z$
	such that $\varphi(X)$ is a unipotent subgroup of $(K \times H) / Z$.
\end{lemma}

\begin{proof}
	Denote $K' = K / K \cap Z$, $H' = H \cap Z \backslash H$ and
	$G' = K \times H / Z$.
	We claim that the projection
	\[
		p \colon G' \to  K' \backslash G'
	\]
	restricts to an embedding on $X$
	after a suitable automorphism of $G'$.
	This can be seen as follows. Let $U \subseteq H$ be a closed
	one-dimensional unipotent subgroup such that
	$\pi \colon G' \to G' / U$ restricts to an embedding on $X$; see
	Remark~\ref{rem.OneDimProj}.
	Thus we have a commutative diagram
	of principal $U$-bundles and $\textrm{pr}_1$ is
	$U$-equivariant:
	\[
		\xymatrix{
			G' \ar[d]^-{\pi} \ar[r]^-{\textrm{pr}_1} & H / \Center(H) \ar[d] \\
			G'/U \ar[r]^-{\textrm{pr}_2} & (H/\Center(H))/U \, ;
		}
	\]
	where $\Center(H)$ denotes the center of $H$.
	This diagram restricts to a $U$-equivariant
	morphism of principal $U$-bundles
	\[
		\xymatrix{
			\pi^{-1}(\pi(X)) \ar[d] \ar[r] & \textrm{pr}_1(\pi^{-1}(\pi(X))) \ar[d] \\
			\pi(X) \ar[r] & \textrm{pr}_2(\pi(X)) \, .
		}		
	\]
	Since ${\C^+}$ is a special group in the sense of
	Serre~\cite[\S4]{Se1958Espaces-fibres-}, both
	principal bundles are locally trivial.
	Since the base varieties $\pi(X)$ and $\pr_2(\pi(X))$ are affine, both
	principal bundles are trivial.
	Since $\pi(X) \cong \AA$, there exists a section
	$Y \subseteq \pi^{-1}(\pi(X))$ over $\pi(X)$ that is mapped
	isomorphically onto its image via $\textrm{pr}_1$.
	By {Proposition~\ref{prop.ConstructionOfAuto}
	there exists an automorphism} of $G$ which moves
	$X$ into $Y$ along the fibers of $\pi$ and thus
	we can assume that $\pr_1$ restricts
	to an embedding on $X$.
	Since $\pr_1 \colon G' \to H / \Center(H)$
	factors through the projection $p$,
	this proves the claim.
	
	Note that $K'$ is normal in $G'$ and therefore $K' \backslash G'$
	is an algebraic group and in particular affine.
	Since $p \colon G' \to K' \backslash G'$
	restricts to an embedding on $X$,
	we can apply Lemma~\ref{lem.twosubgroups}
	to the algebraic group $G'$ and the closed {connected}
	subgroups $K'$, $H'$
	and hence assume that $X \subseteq H'$. Since $K = K^u$
	and $H = H^u$ it follows that $G' = (G')^{u}$. Hence we can
	apply Lemma~\ref{lem.subgroup} to $G'$ and the proper subgroup $H'$
	to get an automorphism $\varphi$ of $G'$ such that $\varphi(X)$
	is a unipotent subgroup of $G'$.
\end{proof}

\begin{proof}[Proof of Proposition~\ref{prop.non-semisimple}]	
Since $G$ is a semisimple algebraic group, there exist
simple algebraic groups $G_1, \ldots, G_n$ and an epimorphism
\[
	G_1 \times \cdots \times G_n \to G
\]
with finite kernel; see \cite[Theorem~27.5]{Hu1975Linear-algebraic-g}.
As $G_1 \times \cdots \times G_n$ is connected,
this kernel is central. By assumption,
$n \geq 2$. If the Lie type of $G$ is equal to $\Lie{sl}_2 \times \Lie{sl}_2$,
then $G$ is isomorphic (as a variety) to one of the groups
\[
	\SL_2 \times \SL_2 \, , \ \SL_2 \times \PSL_2 \ \textrm{or} \
	\PSL_2 \times \PSL_2 \, .
\]
Indeed, if we consider the quotients of $\SL_2 \times \SL_2$ by subgroups of the center
\[
	\Center(\SL_2 \times \SL_2)=\{(E,E),(E,-E),(-E,E),(-E,-E)\} \, ,
\]
we get
\[
	\frac{\SL_2 \times \SL_2}{\langle (E, E) \rangle} \cong \SL_2 \times \SL_2 \, ,
	\quad
	\frac{\SL_2 \times \SL_2}{\Center(\SL_2 \times \SL_2)} \cong
	\PSL_2 \times \PSL_2 \, ,
\]
and
\[
	\frac{\SL_2 \times \SL_2} {\langle(-E,-E)\rangle} \cong
	\frac{\SL_2 \times \SL_2}{\langle(E,-E)\rangle} \cong
	\frac{\SL_2\times\SL_2} {\langle(-E,E)\rangle} \cong
	\SL_2 \times \PSL_2;
\]
where the first isomorphism of the last line is induced by the automorphism
\[
	\SL_2 \times \SL_2\to\SL_2 \times \SL_2, \quad (A,B)\mapsto (AB,B) \, .
\]
By Proposition~\ref{prop.Product} all embeddings of $\AA$ into
one of these groups are equivalent.
Hence, we can assume
that the Lie type of $G$ is not equal to $\Lie{sl}_2 \times \Lie{sl}_2$.
Therefore one can find a semisimple algebraic group $K$ and a simple
algebraic group $H$ such that $G \cong (K \times H) / Z$
for a central finite subgroup $Z$
and the Lie algebra of $H$ is not isomorphic to $\Lie{sl}_2$.
Since $H$ is simple, the classification of simple
Lie algebras implies that $\rank H \geq 2$.
By Lemma~\ref{lem.enoughUnipotentsInAParabolic}, we have
$\dim \UUU_H \geq 4$.
If $X \subseteq (K \times H) / Z$ is a closed curve that is isomorphic to $\AA$,
then we can apply Lemma~\ref{lem.almostproducts}
to $K$, $H$ and the finite central subgroup $Z \subseteq K \times H$
to find an automorphism that maps $X$ into a unipotent subgroup.
Thus Proposition~\ref{prop.firstExample} implies the result.
\end{proof}

\begin{remark}
	Note that $\SL_2 \times \SL_2 / \langle (-E, -E) \rangle$ and
	$\SL_2 \times \PSL_2$ are not isomorphic as \emph{algebraic groups},
	since $(A, B) \mapsto (B, A)$ is an automorphism of the first algebraic group
	that is not inner; however, all automorphisms of the
	second algebraic group are inner, since (by a calculation)
	\[
		\Aut_{\textrm{alg.grp.}}(\SL_2 \times \PSL_2) \cong
		\Aut_{\textrm{alg.grp.}}(\SL_2) \times \Aut_{\textrm{alg.grp.}}(\PSL_2)
	\]
	and since all automorphisms of the algebraic groups
	$\SL_2$, $\PSL_2$ are inner;
	see \cite[Theorem~27.4]{Hu1975Linear-algebraic-g}.
\end{remark}

\section{Embeddings into simple groups}
\label{sec.Simple}

In this section, we {prove} the hardest part of Theorem~\ref{thm:mainthm}:

\begin{theorem}\label{thm:mainthmSimple}
	Let $G$ be a simple algebraic group of rank at least two.
	Then two embeddings of the affine line into
	$G$ are the same up to an automorphism of $G$.
\end{theorem}

We remark that Theorem~\ref{thm:mainthmSimple}, Theorem~\ref{thm.reduction_to_semisimple} and Proposition~\ref{prop.non-semisimple}
imply Theorem~\ref{thm:mainthm}. We do this in detail:

\begin{proof}[Proof of Theorem~\ref{thm:mainthm}]
By Lemma~\ref{lem.Product}, $G$ is isomorphic
to $G^u\times (\C^*)^n$ as a variety,
where $n$ is some non-negative integer. By Lemma~\ref{lem.reduction_to_characterless}, all embeddings of $\C$ into $G$ are equivalent if and only if all embeddings of $\CC$ into $G^u$ are equivalent. Hence, it suffices to consider embeddings of $\CC$ into $G^u$.
If $G^u$ is not semisimple and not isomorphic as a variety to $\C^3$,
then all embeddings of $\C$ into $G^u$ are equivalent by Theorem~\ref{thm.reduction_to_semisimple}. If $G^u$ is semisimple but not simple, then all embeddings of $\C$ into $G^u$ are equivalent by Proposition~\ref{prop.non-semisimple}. Finally, if $G^u$ is simple and different from $\SL_2$ and $\PSL_2$, then $G^u$ has rank at least two; thus all
embeddings of $\C$ into $G^u$ are equivalent by Theorem~\ref{thm:mainthmSimple}.
\end{proof}

\subsection{Outline of the proof of Theorem~\ref{thm:mainthmSimple}}
In the light of Proposition~\ref{prop.firstExample}, it is enough to prove that any
closed curve $X \subseteq G$ which is isomorphic to $\AA$
can be moved into a one-dimensional unipotent subgroup of $G$ via
an automorphism of $G$.
In a first step we move our $X$ into a naturally defined subvariety $E$ (see Section~\ref{sec.MovingIntoE})
and in a second step we move it into a proper
subgroup (see Section~\ref{sec.MovingIntoProperSubgroup}).
By Lemma~\ref{lem.subgroup} we are then able to
move $X$ into a one-dimensional unipotent subgroup of $G$,
which then finishes the proof of Theorem~\ref{thm:mainthmSimple}.

The subvariety $E$ is defined using classical theory of algebraic groups. The necessary notion is set up in the next subsection.

\subsection{Notation and basic facts}
\label{sec.EmbeddingsSimpleNotation}
Let us fix the following notation for the whole section. By
$G$ we denote a simple algebraic group,
by $B \subseteq G$ a fixed Borel subgroup
and by $T \subseteq B$ a fixed maximal torus. Let $\Phi$ be the irreducible
roots system of $G$ with respect to $T$. Moreover, we denote by
$W$ the Weyl group with respect to $T$ and we denote by $\Delta$
the base of $\Phi$ with respect to $B$.
We denote by $w_0$ the unique longest word in $W$ with respect to $\Delta$
and by $B^-$ the opposite Borel subgroup of $B$ that contains $T$, i.e.~$B^- = w_0 B w_0$.

We fix a maximal parabolic subgroup $P$ that contains $B$,
i.e.~we fix a simple root $\alpha \in \Delta$ such that
$P = B W_{I} B$ where $I = \Delta \setminus \{ \alpha\}$
and $W_{I}$ denotes the subgroup in $W$ generated by the
reflections corresponding to the roots in $I$. We denote the reflection
corresponding to $\alpha$ by $s_\alpha$. Furthermore, we denote by
$P^-$ the unique opposite parabolic subgroup to
$P$ {with respect to T; see Appendix~\ref{sec.OppositeParabolicSubgroup}.
Again by Appendix~\ref{sec.OppositeParabolicSubgroup},
$P^- = B^- W_I B^-$, $PP^-$ is open in $G$
and $PP^- = R_u(P)P^- = P R_u(P^-)$.}

{
The quotient of $G$ by the unipotent radical of $P^-$ will play a crucial role for use.
We denote this quotient throughout this section by
\[
	\pi \colon G \to G / R_u(P^-) \, .
\]
Since $R_u(P^-)$ is a special group in the sense of Serre
\cite[\S4]{Se1958Espaces-fibres-},
$\pi$ is a locally trivial principal $R_u(P^-)$-bundle.}

Since $P$ is a maximal parabolic subgroup of $G$, there exists
a unique Schubert curve in $G/P$. We denote
by $E$ the inverse image of this Schubert
curve under the natural projection $G \to G/P$. Note that $E$ is the union of
the two disjoint subsets $B s_\alpha P$ and $P$ of $G$.

\subsection{\texorpdfstring{The restriction of $\pi$ to $E$}
{The restriction of pi to E}}

Recall that $E$ denotes the inverse image of the unique
Schubert curve in $G/P$ and $\pi \colon G \to G / R_u(P^-)$ {denotes}
the canonical projection.
The following result describes the restriction of $\pi$ to $E$.
It is the key ingredient that enables us to move our curve into $E$.

\begin{proposition}
	\label{prop.keyProperteOfPi}
	The complement of $\pi(E)$ in $G/R_u(P^-)$ is
	closed and has codimension at least two in $G/R_u(P^-)$.
	Moreover, the restriction of $\pi$ to $E$ turns $E$
	into a locally trivial $\AA$-bundle over $\pi(E)$.
\end{proposition}


\begin{proof}[Proof of Proposition~\ref{prop.keyProperteOfPi}]
	For the first statement it is enough to show that
	$\pi^{-1}(\pi(E)) = EP^-$ is open in $G$
	and that $G \setminus EP^-$ has codimension at least two in $G$.
	We have the following inclusion inside $G$
	\[
		BP^- \cup B s_\alpha P^- \subseteq PP^- \cup B s_\alpha PP^-
		= EP^- \, .
	\]
	Since $BP^- = PP^-$ is open in $G$, it follows that $EP^-$
	is open in $G$. More precisely, $G \setminus B P^-$
	is an irreducible closed hypersurface in $G$. This follows from the fact that
	$(G/P^-) \setminus Be$ is the translate by $w_0$ of
	the unique Schubert divisor in $G/P^-$ with respect to $B^-$.
	Since $BP^-$ and $B s_\alpha P^-$ are disjoint
	we have a proper inclusion $G \setminus EP^- \subsetneq G \setminus BP^-$.
	Thus $G \setminus EP^-$ has codimension at least two in $G$.
	
	For proving the second statement, we first show
	that all fibers of $\pi |_E \colon E \to \pi(E)$ are reduced and
	isomorphic to $\AA$. In fact, the schematic fiber over $\pi(g)$
	is the schematic intersection $E \cap g R_u(P^-)$ for all $g \in E$.
	Let $C$ be the unique Schubert curve in $G/P$, i.e.~$C$ is the closure of the $B$-orbit through $s_\alpha$.
	Since Schubert varieties are normal
	(see \cite[Theorem~3]{RaRa1985Projective-normali})
	and rational, it follows that $C \cong \PP^1$.
	For each {$g \in G$}, consider
	the following commutative diagram
	\[
		\xymatrix{
			E \cap gR_u(P^{-}) \ar[d] \ar[r] & E
			\ar[d] \ar[r] & C \ar[d] \\
			g R_u(P^-) \ar[r] & G \ar[r] & G/P \, .
		}
	\]
	Note that all squares are
	pull-back diagrams. Since $g R_u(P^{-}) \to G \to G/P$ is an
	open injective immersion,
	the same holds for $E \cap g R_u(P^-) \to E \to C$.
	Note that the image of $g R_u(P^-)$ inside $G/P$ is equal to
	$g B^- e \subseteq G/P$. Since $E$ is the inverse image
	of $C$ under $G \to G/P$ we get an isomorphism
	\[
		E \cap g R_u(P^-) \cong C \cap g B^- e \, .
	\]
	Let $C^{\op} \subseteq G/P$ be the opposite Schubert variety to $C$, i.e.~$C^{\op}$ is the closure of the $B^-$-orbit through $s_\alpha$ inside $G/P$.
	Thus we have a disjoint union
	\[
		C^{\op} \cup B^- e = G/P \, .
	\]
	It follows from Lemma~\ref{lem.alternative} that for all $g \in G$ the
	subset $C \cap gC^{\op}$ consists of exactly one point
	or $C \subseteq gC^{\op}$. Hence
	\[
		C \setminus (C \cap g C^{\op}) =
		C \cap g B^- e \, .
	\]
	is either isomorphic to $\AA$ or it is empty. This proves
	that all fibers of $\pi |_E \colon E \to \pi(E)$
	are reduced and isomorphic to $\AA$.
	
	Since $C$ is smooth and since $G \to G/P$ is a smooth morphism,
	it follows that
	$E$ is smooth; see
	\cite[Chp. II, Proposition~3.1]{GrRa2004Revetements-etales}.
	Moreover, $\pi(E)$ is smooth as an open subset
	of the smooth variety $G/R_u(P^-)$. Since all fibers of $\pi |_E$
	have the same dimension, the morphism $\pi |_E$ is faithfully flat.
	Since $\pi$ is affine as a locally trivial principal
	$R_u(P^-)$-bundle, the restriction $\pi |_E$ is  also affine.
	It follows from \cite{KaWr1985Flat-families-of-a} or
	\cite[Theorem~5.2]{KrRu2014Families-of-group-} that
	$\pi |_E$ is a locally trivial $\AA$-bundle.
\end{proof}

\begin{lemma}	
	\label{lem.alternative}
	Let $C$ be the unique Schubert curve in $G/P$ with respect to $B$
	and let $C^{\op}$ be the opposite Schubert variety to $C$.
	Then for all $g \in G$ either $gC \cap C^{\op}$ is a reduced
	point of $G/P$ or
	$gC \subseteq C^{\op}$.
\end{lemma}

\begin{remark}
	Compare the proof of this lemma with
	\cite[Chp. III, Proof of Theorem~10.8]{Ha1977Algebraic-geometry}.
\end{remark}

\begin{proof}
	Consider the following pullback diagram
	\[
		\xymatrix{
			(G \times C) \times_{G/P} C^{\op} \ar[d]
			\ar[r] & C^{\op} \ar[d] \\
			G \times C \ar[r] & G/P
		}
	\]
	where $G \times C \to G/P$ denotes the map
	$(g, c) \mapsto gc$.
	Note that the vertical arrows are closed embeddings.
	Since $C$ is smooth, by generic smoothness
	\cite[Chp. III, Corollary~10.7]{Ha1977Algebraic-geometry}
	and $G$-equivariance, the morphism $G \times C \to G/P$ is smooth.
	Since $C^{\op}$ is reduced, it follows that the fiber product
	$(G \times C) \times_{G/P} C^{\op}$ is reduced
	\cite[Chp. II, Proposition~3.1]{GrRa2004Revetements-etales}.
	Let $q$ be the following composition
	\[
		q \colon (G \times C) \times_{G/P} C^{\op} \to G \times C \to G
	\]
	where the last map is the projection to the first factor. Note that the fiber of $q$
	over $g \in G$ is isomorphic to the scheme theoretic
	intersection $gC \cap C^{\op}$. Since $C$ is projective, the morphism
	$q$ is projective and thus by \cite[Theorem~14.8]{Ei1995Commutative-algebr}
	the subset
	\[
		V = \{ \, g \in G \ | \ \textrm{$gC \cap C^{\op}$ is finite} \, \}
	\]
	is open in $G$. Let $q' = q |_{q^{-1}(V)} \colon q^{-1}(V) \to V$.
	By definition, $q'$ is quasi-finite.
	Since $q$ is projective (and thus $q'$ also), it follows
	that $q'$ is finite; see
	\cite[Th\'eor\`eme~8.11.1]{Gr1966Elements-de-geomet-IV_3}.
	We claim that $q$ is birational.
	Indeed, this can be seen as follows. The fiber of $q$ over
	{$e \in G$
	is isomorphic to $C \cap C^{\op}$}.
	By \cite[Theorem~3 and Remark~3]{Ra1985Schubert-varieties}
	this last scheme is reduced and by
	\cite[Theorem~3.7]{Ri1992Intersections-of-d} it is
	irreducible and of dimension
	zero; cf.~also \cite{BrLa2003A-geometric-approa}.
	Thus the fiber of $q$ over $e$ is a reduced point.
	Hence, the tangent
	space of the fiber satisfies
	\[
		0 = T_{x_0} q^{-1}(e) = \ker d_{x_0} q
	\]
	where $\{ x_0 \} = q^{-1}(e)$. Therefore $q$ is immersive at $x_0$.
	Hence $q^{-1}(V)$ is smooth at $x_0$ by dimension reasons
	and $q'$ is \'etale in $x_0$.
	Let $S$ be the set of points in $q^{-1}(V)$, where
	$q'$ is not \'etale. By
	\cite[Chp. I, Proposition~4.5]{GrRa2004Revetements-etales} the set $S$
	is closed in $q^{-1}(V)$.
	As $q'$ is finite, $q(S)$ is closed in $V$. Clearly, $q'$
	restricts to a finite \'etale morphism
	\begin{equation}
		\label{eq.restr}
		q^{-1}(V \setminus q(S)) \longrightarrow V \setminus q(S) \, .
	\end{equation}
	Since $\{ x_0 \}$
	is a fiber of $q$ and since $q'$ is \'etale at $x_0$, it follows that
	$q(x_0) \not\in q(S)$, i.e.~$x_0 \in q^{-1}(V \setminus q(S))$.
	This implies that the morphism~\eqref{eq.restr} is of degree one
	{and therefore it is an isomorphism.
	Since $V$ is irreducible and since $q'$
	is finite, it follows that $q^{-1}(V \setminus q(S))$ is dense in $q^{-1}(V)$.
	As~\ref{eq.restr} is an isomorphism, $q^{-1}(V)$ is irreducible.
	This implies that $q'$ is birational.}
	Since $V$ is smooth and irreducible and since $q'$ is finite and birational,
	it follows that $q'$ is an isomorphism by Zariski's Main Theorem
	\cite[Corollaire~4.4.9]{Gr1961Elements-de-geomet-III_1}).
	This implies the lemma.
\end{proof}

\subsection{Moving a curve into $E$}
\label{sec.MovingIntoE}

\begin{proposition}
	If $X \subseteq G$ is a closed curve that is isomorphic to $\AA$,
	then there exists an automorphism $\varphi$ of $G$ such that
	$\varphi(X) \subseteq E$.
\end{proposition}

\begin{proof}
	If $\rank(G) = 1$, then $E = G$ and there is nothing to prove.
	Thus we assume that $\rank(G) \geq 2$. Therefore, we we can
	apply Proposition~\ref{prop.UnipotProj} to $G$ and the unipotent
	subgroup $R_u(P^-)$ to get an automorphism
	$\varphi$ of $G$ such that $\pi \colon G \to G / R_u(P^-)$
	restricts to an embedding on $\varphi(X)$.
	Let us replace $X$ by $\varphi(X)$.
	Since the complement of $\pi(E)$ in $G/R_u(P^-)$ is closed and
	has codimension
	at least two in $G/R_u(P^-)$ by Proposition~\ref{prop.keyProperteOfPi},
	there exists by Kleiman's Theorem $g \in G$ such that $g \pi(X)$
	lies inside $\pi(E)$; see~\cite[Theorem~2]{Kl1974The-transversality}.
	Since $\pi$ is $G$-equivariant, it restricts to an isomorphism
	$g X \to \pi(gX)$. Hence, we can replace $X$ by $gX$
	and assume in addition that $\pi(X) \subseteq \pi(E)$.
	Since $\pi$ restricts to a locally trivial $\AA$-bundle
	$\pi |_E \colon E \to \pi(E)$
	by Proposition~\ref{prop.keyProperteOfPi} and since $\pi(X) \cong \AA$,
	there exists a section $\sigma$ of $\pi |_E$ over $\pi(X)$;
	see e.g.~\cite{BaCoWr1976Locally-polynomial}.
	By {Proposition~\ref{prop.ConstructionOfAuto}}
	there exists an automorphism
	of $G$ that moves $X$ to the section $\sigma(\pi(X))\subset E$
	and fixes $\pi \colon G \to G / R_u(P^-)$. This implies the result.
\end{proof}

\subsection{Moving a curve in $E$ into a proper subgroup}
\label{sec.MovingIntoProperSubgroup}

The aim of this section is to prove the following result.

\begin{proposition}
	\label{prop.MovingIntoProperSubgroup}
	Assume that $\rank{G} \geq 2$.
	If $X \subseteq E$ is a closed curve that is isomorphic to $\AA$,
	then there exists an automorphism $\varphi$
	of $G$ such that $\varphi(X)$ lies in a proper subgroup of $G$.
\end{proposition}

Proposition~\ref{prop.MovingIntoProperSubgroup} is based on the following rather technical result.

\begin{proposition}
\label{prop.key}
Assume that $\rank(G) \geq 2$.
Let $K$ be a closed connected reductive subgroup of $G$ such that
$KP$ is closed in $G$.
Assume that $K \cap P$ is connected and solvable and moreover,
that $R_u(K \cap P)$ has dimension one and lies in $R_u(P)$.
If $X \subseteq KP$ is a closed curve that is isomorphic to $\AA$, then
there exists an automorphism $\varphi$ of $G$
such that $G \to K \backslash G$ restricts to an
embedding on $\varphi(X)$ {and $\varphi(KP) = KP$.}
\end{proposition}

Before proving Proposition~\ref{prop.key}, we show how it implies
Proposition~\ref{prop.MovingIntoProperSubgroup}.

\begin{proof}[Proof of Proposition~\ref{prop.MovingIntoProperSubgroup}]
	Let $K = \Cent_G((\ker \alpha)^\circ)$ be the centralizer in $G$
	of the connected component of the identity element
	of the kernel of the root
	$\alpha \colon T \to {\C^*}$.
	By definition, $T$ and the root subgroups $U_{\pm \alpha}$ lie inside $K$.
	By \cite[Theorem~22.3, Corollary~26.2B]
	{Hu1975Linear-algebraic-g},
	the group $K$ is connected, reductive, the semisimple rank is one
	and the Lie algebra of $K$ decomposes as
	$\Lie{t} \oplus \Lie{u}_{\alpha} \oplus \Lie{u}_{-\alpha}$,
	where $\Lie{t}$ is the Lie algebra of $T$ and $\Lie{u}_{\pm \alpha}$
	is the Lie algebra of $U_{\pm \alpha}$.
	Since $K$ is connected and not solvable,
        $T U_\alpha$ is connected and solvable,
	and $T U_\alpha$ is of codimension one in $K$, it follows that
	$T U_\alpha$ is a Borel subgroup of $K$. Since
	$T U_\alpha \subseteq K \cap P \subseteq K$, the subgroup
	$K \cap P$ is parabolic in $K$ and in particular
	it is connected; see \cite[Corollary~23.1B]{Hu1975Linear-algebraic-g}.
	We have $K \cap P \neq K$, since otherwise $P$ would
	contain the root subgroup $U_{- \alpha}$
	and thus we would have $P = G$; see~\cite[Theorem~27.3]{Hu1975Linear-algebraic-g}. Hence
	\[
		K \cap P = T U_\alpha \, .
	\]
	Moreover, we have by \cite[\S30.2]{Hu1975Linear-algebraic-g}
	\[
		R_u(K \cap P) = U_\alpha \subseteq R_u(P) \, .
	\]
	We claim that $U_\alpha s_\alpha P = B s_\alpha P$
	inside $G$. Indeed,
	otherwise $U_\alpha s_\alpha P = s_\alpha P$,
	since $\dim B s_\alpha P = \dim E = 1 + \dim P$.
	Therefore $U_{-\alpha} = s_\alpha U_\alpha s_\alpha \subseteq P$,
	a contradiction. Hence it follows that
	\[
		E = U_\alpha s_\alpha P \cup P \, .
	\]
	Since $T$ and $U_{\pm\alpha}$ generate $K$,
	it follows that $K$ lies inside the minimal parabolic
	subgroup $P_{\{ \alpha \}} = B s_\alpha B \cup B$.
	By \cite[Theorem~13.18]{Bo1991Linear-algebraic-g}
	the reflection $s_\alpha$ generates the Weyl group of
	$K$ and, in particular,
	$s_\alpha$ lies in $K$. More precisely, every representative
	of $s_\alpha$ lies in $K$.
%
%
	In summary, we get
	\[
		E \subseteq KP \subseteq
		P_{\{ \alpha \}} P = B s_\alpha P \cup P = E \, ,
	\]
	which proves $E = KP$.
	
	Now, we can apply
	Proposition~\ref{prop.key} and thus we can
	assume that $G \to K \backslash G$ restricts to an embedding
	on $X$. Applying Lemma~\ref{lem.twosubgroups} to $G$
	and the {closed connected} subgroups
	$K$ and $P$ yields the desired result.
\end{proof}

The rest of this subsection is devoted to the proof of
Proposition~\ref{prop.key}. First we provide an estimation of the
dimension of the  intersection of every translate of the torus $T$ with
the variety $\UUU_G$ of unipotent elements in case $G$ is of rank two.
Note that by the classification of simple groups of rank two,
$G$ is either of type $A_2$, $B_2$ or $G_2$.

\begin{lemma}
	\label{lem.rank2}
	Assume that $\rank(G) = 2$. Then the following holds.
	\begin{enumerate}[$i)$]
		\item {If $G$ is of type $A_2$, then
			$T p \cap \UUU_G$ is finite for all $p \in P$.}
		\item If $G$ is of type $B_2$, then
			{$\dim( T g \cap \UUU_G) \leq 1$} for all $g \in G$.
	\end{enumerate}
\end{lemma}

\begin{remark}
	To complement $i)$, note that for some $g \in G$
	the intersection $T g \cap \UUU_G$ is not finite.
	For example, if $G = \SL_3$, $T$ is the diagonal torus in $G$ and
	\[
		g =
		\begin{pmatrix}
			 3 & 0 & -4Ê\\
			 2 & 0 & -3 \\
			 0 & 1 & 0
		\end{pmatrix} \, ,
	\]
	then a calculation shows that $T g \cap \UUU_G$ is one-dimensional.
\end{remark}

\begin{proof}[Proof of Lemma~\ref{lem.rank2}]
	To every simple group $H$ there exists a simply connected
	simple group $\tilde{H}$ and an isogeny
	$\tilde{H} \to H$, i.e.~an epimorphism with finite kernel;
	see~\cite[\S23.1, Proposition~1]{Ch2005Classification-des}.
	Two simply connected simple groups with the same root
	system are always isomorphic by
	\cite[Theorem~32.1]{Hu1975Linear-algebraic-g}.
	Therefore it is enough to prove $i)$ for the simply connected
	group $G = \SL_3$ and to prove
	$ii)$ for the simply connected group
	$G = \Sp_4$;
	{see~\cite[\S20.1, \S22.1]{Ch2005Classification-des}
	and~\cite[Corollary~21.3C]{Hu1975Linear-algebraic-g}.}
	
	Assume $G$ is $\SL_3$. We can assume that $T$
	is the subgroup of $G$ of diagonal matrices and $B$ is the subgroup
	of upper triangular matrices. Moreover, we can assume without loss
	of generality that $P$ is the maximal parabolic subgroup
	\[
		P = \left\{ \begin{pmatrix}
			 \ast & \ast & \astÊ\\
			 0 & \ast & \ast \\
			 0 & \ast & \ast
		       \end{pmatrix} \right\}
		       \subseteq \SL_3 \, .
	\]
	An element $a \in \SL_3$ is unipotent if and only
	if one is the only root of its characteristic polynomial $\chi_a$.
	We have
	\[
		\chi_a(t) = t^3 - \tr(a) t^2 + s(a) t - 1
	\]
	where
	\[
		s(a) = (a_{11} a_{22} - a_{12} a_{21})
			+ (a_{11} a_{33} - a_{13} a_{31})
			+ (a_{22} a_{33} - a_{23} a_{32})
	\]
	and $a_{ij}$ denotes the $ij$-th entry of $a$.
	Let $p \in P$.
	The variety $T p \cap \UUU_G$ is isomorphic to
	\[
		S = \{ \, t \in T \ | \ tp \in \UUU_G \, \}  \, .
	\]
	Let $x, y, z$ denote the entries on the diagonal of a
	$3 \times 3$-diagonal matrix.
	The set $S$ can be realized as the closed subvariety
	of $\AA^3$ given by the equations
	\begin{align}
		3 &= x p_{11} + y p_{22} + z p_{33} \label{EQ1} \\
		3 &= xy p_{11} p_{22} +
			xzp_{11} p_{33} +
			yz(p_{22} p_{33} - p_{23} p_{32}) \label{EQ2} \\
		1 &= xyz \label{EQ3} \, .
	\end{align}
	Clearly, $p_{11}$ is non-zero. Inserting \eqref{EQ1} in \eqref{EQ3} yields
	the irreducible equation
	\begin{equation}
		\label{eq.FirstEquation}
		p_{11} = (3- y p_{22} - z p_{33})yz \, .
	\end{equation}
	Inserting \eqref{EQ1} in \eqref{EQ2} yields a non-trivial
	equation of degree $\leq 2$
	in $y$ and $z$. If $p_{22}$ or $p_{33}$
	is non-zero, then \eqref{eq.FirstEquation}
	is an equation of degree $3$ and thus $S$ is finite.
	If $p_{22} = p_{33} = 0$, then $S$ is realized as the closed subset of $\CC^2$
	given by the equations
	\[
		3 = - yz p_{23} p_{32} \et p_{11} =3yz \, .
	\]
	However, since $p$ has determinant equal to $1$, we get
	$-p_{11} p_{23} p_{32} = 1$. Hence, $S$ is empty
	in case $p_{22} = p_{33} = 0$. This proves $i)$.

	Assume that $G$ is $\Sp_4$. Since all non-degenerate
	alternating bilinear forms on an even dimensional vector space are
	equivalent, we can choose
	$\Omega$ as the matrix with entries $1$, $1$, $-1$, $-1$
	on the antidiagonal and all other entries equal to zero,
	and then define $\Sp_4$ as those $4 \times 4$-matrices $g$
	that satisfy $g^t \Omega g = \Omega$.
	Thus we can choose for the maximal torus $T$ the subgroup
	of $\Sp_4$ consisting of diagonal matrices with entries
	$t_1$, $t_2$, $t_2^{-1}$, $t_1^{-1}$ on the diagonal for arbitrary non-zero
	$t_1$ and $t_2$.
	If an element in $\GL_4$ is unipotent, then its trace is equal to $4$.
	Let $g \in \Sp_4$. One can see that
	\[
		\{ \, t \in T \ | \ \tr{tg} = 4 \, \}
	\]
	is a proper closed subset of the torus $T$ and thus
	$Tg \cap \UUU_{\Sp_4}$ is properly contained in $Tg$,
	which proves $ii)$.
\end{proof}

\begin{lemma}
	\label{lem.fiberDimension}
	Assume that $\rank(G) \geq 2$.
	Let $H \subseteq P$ be a connected closed solvable subgroup
	such that the unipotent radical $R_u(H)$ is one-dimensional.
	Denote by $\rho \colon P \to H \backslash P$ the canonical projection.
	Then for every $p \in P$ the fibers of the morphism
	\[
		\UUU_P \to \rho(p) \UUU_P \, , \quad u \mapsto \rho(p) u
	\]
	have codimension at least three in $\UUU_P$.
\end{lemma}

\begin{proof}
	In case the rank of $G$ is at least $3$ or $G$ is of type
	$G_2$, it follows that
	\[
		\dim \UUU_P - \dim H \geq \dim \UUU_P - \dim T - 1
		\geq 3
	\]
	by Lemma~\ref{lem.enoughUnipotentsInAParabolic}, and thus the lemma
	is proved in these cases.
	
	Assume that $G$ is of type $A_2$.
	{For every $p \in P$ the quotient
	$\eta \colon P \to T \backslash P$ restricts
	to a morphism $Hp \cap \UUU_P \to \eta(R_u(H)p)$.
	By Lemma~\ref{lem.rank2} the fibers of this restriction are finite.
	Since $R_u(H)$ is one-dimensional, it follows that
	$H p \cap \UUU_P$ is at most one-dimensional.
	By Lemma~\ref{lem.enoughUnipotentsInAParabolic},
	we have $\dim \UUU_P = 4$, which implies the lemma in this case.}
	
	Assume that $G$ is of type $B_2$.  Analogously, it follows from
	Lemma~\ref{lem.rank2} and Lemma~\ref{lem.enoughUnipotentsInAParabolic}
	that $H p \cap \UUU_P$ is at most two-dimensional and that
	$\dim \UUU_P = 5$, which proves the lemma in this case.	
\end{proof}

\begin{proof}[Proof of Proposition~\ref{prop.key}]
	We start by observing that $K \cap P \backslash P$ is affine,
	since
	$K \cap P \backslash P \cong K \backslash KP$ is closed in
	$K \backslash G$ and since $K \backslash G$
	is affine ($K$ is reductive).
	In particular, every ${\C^+}$-orbit of a ${\C^+}$-action on
	$K \cap P \backslash P$ is closed.
	
	We proof that for a generic $u \in \UUU_P$
	the one-dimensional unipotent subgroup ${\C^+}(u)$ of $P$
	acts without fixed point on $K \cap P \backslash P$.
	Every ${\C^+}(u)$-orbit in $K \cap P \backslash P$ is either a fixed
	point or isomorphic to $\C$.
	If $p \in P$ would map to a fixed point in $K \cap P \backslash P$
	of the $\C^+(u)$-action, then
	$(KÊ\cap P) p {\C^+}(u) = (K \cap P) p$. This would imply that
	$p {\C^+}(u) p^{-1} \subseteq K \cap P$.
	Since $K \cap P$ is solvable, $p {\C^+}(u) p^{-1}$
	lies inside $R_u(K \cap P)$ and hence inside
	$R_u(P)$, by assumption. In particular, ${\C^+}(u)$ lies inside
	$R_u(P)$. However, generic $u \in \UUU_P$
	are not contained in $R_u(P)$, since $P$ is not a Borel subgroup of $G$.
	This proves our claim.
	
	Denote by $\eta \colon KP \to K \cap P \backslash P$
	the restriction of the canonical projection
	$G \to K \backslash G$.
	By Lemma~\ref{lem.enoughUnipotentsInAParabolic}
	we have $\dim \UUU_P \geq 4$ and hence there exists
	a one-dimensional unipotent subgroup $U$ of $P$
	such that $G \to G /U$ restricts to an embedding
	on $X$, by Remark~\ref{rem.OneDimProj}.
	Moreover, we can assume by the previous paragraph that
	$U$ acts without fixed point on $K \cap P \backslash P$.
	Thus we can apply Lemma~\ref{lem.section_birational1}
	to the $U$-equivariant morphism $XU \to \overline{\eta(XU)}$ to get
	a section $X'$ of $XU \to XU /U$ that is mapped birationally via $\eta$
	onto its image. Hence, after applying an appropriate automorphism
	of $G$ (that leaves $KP$ invariant), we can assume that $\eta$
	maps $X$ birationally onto its image;
	see Proposition~\ref{prop.ConstructionOfAuto}.
	Let us denote this image
	inside $\overline{\eta(XU)}$ by $C$. Note that
	$C$ is closed in $\overline{\eta(XU)}$, since $X$ is isomorphic to
	$\AA$. We apply
	Lemma~\ref{lem.gen-projection-birational} to the group $P$,
	the affine
	homogeneous $P$-space $K \cap P \backslash P$ and the
	curve $C$ in
	$K \cap P \backslash P$ (the codimension assumptions of
	Lemma~\ref{lem.gen-projection-birational} are guaranteed by
	Lemma~\ref{lem.fiberDimension}).
	Thus we get a $u' \in \UUU_P \setminus \{ e\}$
	such that $G \to G / {\C^+}(u')$ restricts to an embedding on $X$
	(by Remark~\ref{rem.OneDimProj}),
	${\C^+}(u')$ acts without fixed point on $K \cap P \backslash P$
	and $S_{u'} \to S_{u'} \aquot {\C^+}(u')$
	restricts to a birational morphism on $C$. Here $S_{u'}$
	denotes the closure of all the ${\C^+}(u')$-orbits in $K \cap P \backslash P$
	that pass trough $C$. 	
	Since $X$ is mapped birationally
	onto $C \subseteq S_{u'}$ and since $C$ is mapped
	birationally onto $S_{u'} \aquot {\C^+}(u')$ it follows that
	$\eta$ restricts to a birational map $X{\C^+}(u') \to S_{u'}$.
	Hence we can apply Lemma~\ref{lem.section_embedd}
	to the ${\C^+}(u')$-equivariant morphism $X{\C^+}(u') \to S_{u'}$
	and get a section $X''$ of $X{\C^+}(u') \to X{\C^+}(u') / {\C^+}(u')$
	that is mapped isomorphically via $\eta$ onto its image inside
	$S_{u'} \subseteq K \cap P \backslash P$.
	By Proposition~\ref{prop.ConstructionOfAuto}
	there exists an automorphism of $G$ (that leaves $KP$
	invariant) and maps $X$ to $X''$
	and thus we can assume that $\eta$ maps $X$
	isomorphically onto $K \cap P \backslash P$. Since
	$\eta$ is the restriction of $G \to K \backslash G$ to $KP$,
	this finishes the proof.
\end{proof}


\appendix

\section{Principal bundles over the affine line}
\label{sec.PrincipalBundlesOverA1}

In \cite{RaRa1984Principal-bundles-} it is stated by referring on
\cite{St1965Regular-elements-o} and \cite{Ra1983Deformations-of-pr},
that {over an algebraically closed field}
every principal $G$-bundle over the affine line is trivial if $G$
is a connected algebraic group.
However, the connectedness assumption is in fact superfluous
over an algebraically closed field of characteristic zero.
For the sake of completeness we give a prove of this result.

\begin{theorem}
	\label{thm.trivial_over_A1}
	Let $G$ be any algebraic group. Then
	every principal $G$-bundle over the affine line $\AA$ is trivial.
\end{theorem}

Before starting with the proof, let us recall a very important construction
that associates a fiber bundle $P \times^G F \to X$ to a principal
$G$-bundle $\pi \colon P \to X$ and a variety $F$ with a left $G$-action
(see \cite[Proposition~4]{Se1958Espaces-fibres-}): the
variety $P \times^G F$ is defined as the quotient of
$P \times F$ by the right $G$-action
\[
	(p, f) \cdot g = (pg, g^{-1}f)
\]
{and the} canonical map $P \times^G F \to X$
is a bundle with fiber $F$ which becomes locally trivial after
a finite \'etale base change, see \cite[Example~c), \S3.2]{Se1958Espaces-fibres-}.

\begin{proof}[Proof of Theorem~\ref{thm.trivial_over_A1}]
	Let $P \to \CC$ be a principal $G$-bundle.
	Let $G^0$ be the connected component of the identity element in $G$.
	The principal $G$-bundle factorizes as
	\[
		P \longrightarrow P \times^G G / G^0 \longrightarrow \AA \, .
	\]
	The first morphism is a principal $G^0$-bundle by~\cite[Proposition~8]{Se1958Espaces-fibres-}.
	The second morphism is a principal
	$G/G^0$-bundle and since $G/G^0$ is finite, it is a finite morphism;
	see~\cite[Proposition~5 and \S3.2, Example~a)]{Se1958Espaces-fibres-}.
	Since the base is $\AA$, this second principal bundle
	admits a section $s \colon \AA \to P \times^G G /G^0$
	(which follows from
	Hurwitz's Theorem \cite[Chp. IV, Corollary~2.4]{Ha1977Algebraic-geometry}).
	Due to Theorem~\ref{thm.TrivialPrincipalBundle},
	the principal $G^0$-bundle $P \to P \times^G G/G^0$
	is trivial over $s(\AA)$, and thus $P \to \AA$ admits a section, which proves
	the Theorem.
\end{proof}

{The main step in the following Theorem is due to Steinberg
\cite{St1965Regular-elements-o}.}

\begin{theorem}
	\label{thm.TrivialPrincipalBundle}
	Let $G$ be a connected algebraic group. Then, every
	principal $G$-bundle over a smooth affine rational curve is trivial.
\end{theorem}

\begin{proof}
	{Let $X$ be a smooth affine rational curve and}
	let $E \to X$ be a principal $G$-bundle.

	First we prove that $E \to X$ admits a section that is defined over
	some open subset of $X$.
	By definition there exists a finite \'etale map from
	an affine curve $U'$ onto an open subset $U$ of the curve $X$
	such that the pull back $E_{U'} \to U'$ is a trivial principal $G$-bundle.
	Let $K$ be the function field of $U$ and let $K'$ be the function field of $U'$.
	We can assume that the field extension $K' /K$ is finite and Galois, by~\cite[\S1.5]{Se1958Espaces-fibres-}. Let $\Gal(K' /K)$ denote the
	Galois group of this extension. We denote by $G(K')$
	the $K'$-rational points of $G$, i.e.~the group of rational maps
	$U' \dashrightarrow G$.
	By~\cite[\S2.3b)]{Se1958Espaces-fibres-} it follows that
	the first Galois cohomology set
	\[
		H^1(\Gal(K'/K), G(K'))
	\]
	describes the isomorphism classes
	of principal $G$-bundles that are defined over some non-specified
	open subset of $U$ such that their pull back via
	$U' \to U$ admit a section over some open $\Gal(K'/K)$-invariant
	subset of $U'$. Hence it is enough to prove that
	$H^1(\Gal(K'/K), G(K'))$ is trivial.
	Let $\bar{K}$ be an algebraic closure of $K$
	that contains $K'$. By
	\cite[\S5.8, Chp.~I]{Se1994Cohomologie-galois},
	the natural map
	\[
		H^1(\Gal(K'/K), G(K')) \to H^1(\Gal(\bar{K}/K), G(\bar{K}))
	\]
	is injective. Note, that $G(\bar{K})$ is an algebraic group
	over $\bar{K}$. Since $K'$
	has transcendence degree one over the ground field,
	the so-called (cohomological)
	dimension of $K'$ is at most one by
	\cite[Example~b), \S3.3, Chp. II]{Se1994Cohomologie-galois}.
	Now, by a result of Steinberg,
	$H^1(\Gal(\bar{K}/K), G(\bar{K}))$ is trivial;
	see~\cite[Theorem~1.9]{St1965Regular-elements-o}.
	Hence, $E \to X$ admits a section over some open subset of $X$.
	
	The principal $G$-bundle $E \to X$ decomposes as
	\[
		E \to E \times^G G/B \to X
	\]
	where the first morphism is a principal $B$-bundle and the second morphism
	is a $G/B$-bundle, locally trivial in the \'etale topology. 	
	Since $E$ becomes trivial over some open subset $V$ of $X$, it follows that
	$E \times^G G/B$ becomes also trivial over $V$, hence
	there exists a rational section $s \colon X \dashrightarrow E \times^G G/B$
	that is defined over $V$.
	Since $G/B$ is projective, to every point $x$ in $X$
	there is a finite \'etale map $f_x$ onto an open neighbourhood of $x$
	such that the pull back of $E \times^G G/B \to X$ via $f_x$
	is projective. This implies that $E \times^G G/B \to X$ is universally closed
	and hence proper.
	Since $X$ is a smooth curve, it follows by the Valuative Criterion
	of Properness that the section $s$ is defined on the whole $X$;
	see \cite[Theorem~4.7, Chp.~II]{Ha1977Algebraic-geometry}.
	Thus the restriction of the {principal} $B$-bundle $E \to E \times^G G/B$
	to $s(X)$ is trivial by Proposition~\ref{prop.trivialSolvable},
	since $X$ has a trivial Picard group. Hence, we proved that
	$E \to X$ admits a section, which implies the statement of the theorem.
\end{proof}

\begin{proposition}
	\label{prop.trivialSolvable}
	{Let $G$ be a connected, solvable algebraic group.
	Then, every principal $G$-bundle
	over any affine variety with vanishing Picard group
	is trivial.}
\end{proposition}

\begin{proof}
	{Let $X$ be an affine variety.}
	By \cite[Proposition~14]{Se1958Espaces-fibres-} every principal
	$G$-bundle is locally trivial, since $G$ is {connected and} solvable.
	Note that the first \v{C}ech cohomology
	\[
		\check{H}^1(X, \underbar{G})
	\]
	is a pointed set that corresponds to the isomorphism classes of
	locally trivial principal $G$-bundles over $X$, where $\underbar{G}$
	denotes the sheaf of groups on $X$ with sections over an open subset
	$U \subseteq X$ {being} the morphisms $U \to G$; see
	\cite[\S3]{Fr1957Cohomologie-non-ab} and \cite[\S3]{Se1958Espaces-fibres-}.
	Since $G$ is solvable {and connected},
	there exists a semidirect product decomposition
	{$G = U \rtimes T$} for a torus $T$ and a unipotent group $U$.
	The short exact sequence corresponding to this decomposition
	yields an exact sequence in cohomology
	\[
		\check{H}^1(X, \underbar{U}) \to \check{H}^1(X, \underbar{G})
		\to \check{H}^1(X, \underbar{T}) \, ;
	\]
	{see \cite[Th\'eor\`eme~I.2]{Fr1957Cohomologie-non-ab}.}
	However, by using a decreasing chain of closed normal subgroups of $U$
	such that each factor is isomorphic to ${\C^+}$
	and by using that
	$\check{H}^1(X, \underbar{${\C^+}$}) = H^1(X, \OOO_X)$ is trivial (since $X$
	is affine) it follows that $\check{H}^1(X, \underbar{U})$ is trivial.
	Since the Picard group
	$\check{H}^1(X, \underbar{${\C^*}$}) = H^1(X, \OOO_X^\ast)$
	vanishes it follows analogously
	that $\check{H}^1(X, \underbar{T})$ is trivial, whence
	$\check{H}^1(X, \underbar{G})$ is trivial. This implies the proposition.
\end{proof}

\begin{remark}
	\label{rem.Triviality_of_affine_bundles}
	The proof of Proposition~\ref{prop.trivialSolvable}
	shows the following. If $G$ is unipotent, then
	every principal $G$-bundle over any affine variety is trivial.
\end{remark}

\section{Generalities on parabolic subgroups}
\label{sec.Genonparabsubgroups}
Throughout this appendix we fix the following notation.
Let $G$ be a connected reductive algebraic group, $B$ a Borel subgroup, $T$
a maximal torus in $B$ and $W$ the Weyl group with respect
to $T$. Moreover, we denote by $\Delta$ the set of simple roots
of $G$ with respect to $(B, T)$.

\subsection{The opposite parabolic subgroup}
\label{sec.OppositeParabolicSubgroup}
Let $P$ be a parabolic subgroup
that contains $B$, i.e.~$P = B W_I B$
where $I$ is a subset of $\Delta$ and $W_I$ is
the subgroup of $W$ generated by the reflections corresponding to roots
in $I$. There exists a unique parabolic subgroup $P^-$ that contains $T$
such that $P \cap P^-$ is a Levi factor of $P$ and $P^-$, i.e.~there
are semidirect product decompositions
\[
	P = R_u(P) \rtimes (P \cap P^-) \quad \textrm{and} \quad
	P^- = R_u(P^-) \rtimes (P \cap P^-) \, ;
\]
see \cite[Corollary 8.4.4.]{Sp2009Linear-algebraic-g}
and \cite[Proposition 14.21]{Bo1991Linear-algebraic-g}.
We call $P^-$ the \emph{opposite parabolic subgroup} of $P$ with respect to $T$.
In fact we can describe
$P^-$ as follows.
\begin{lemma}
	\label{lem.opposite_parabolic}
	We have $P^- = B^- W_I B^-$.
\end{lemma}

For the lack of reference, we provide a proof.

\begin{proof}
	Let $Z$ be the connected component of the identity element in
	the group
	$\bigcap_{\gamma \in I} \ker \gamma$. By
	\cite[\S30.2]{Hu1975Linear-algebraic-g}, the centralizer
	$\Cent_G(Z)$ is a Levi factor of $P$,
	i.e.~$P = \Cent_G(Z) \ltimes R_u(P)$. Let $Q$ be $B^- W_I B^-$.
	In fact, $Q = B^- W_{-I} B^-$ since $W_I = W_{-I}$.
	Moreover, $Z$ is the connected component of the identity element in
	$\bigcap_{\gamma \in -I} \ker \gamma$
	and thus it follows that $Q = \Cent_G(Z) \ltimes R_u(Q)$.
	Clearly, $R_u(Q) \cap P$ is a unipotent subgroup of $R_u(Q)$
	that is invariant under conjugation by $T$. If
	$R_u(Q) \cap P$ is non-trivial, it contains a root subgroup $U_{\beta}$, by
	\cite[Proposition~28.1]{Hu1975Linear-algebraic-g}) for a certain
	root $\beta$.
	Note that $\beta$ is a negative root with respect to $\Delta$ which is not a
	$\ZZ$-linear combination of roots in $I$, by
	\cite[\S30.2]{Hu1975Linear-algebraic-g} applied to $(B^-, Q)$.
	Since $\beta$ is also a root of $P$ with respect to $T$, we get
	a contradiction to \cite[Proposition~30.1]{Hu1975Linear-algebraic-g}
	applied to $(B, P)$. Hence $R_u(Q) \cap P$ is trivial. This implies that
	$P \cap Q = \Cent_G(Z)$ and thus $P$ and $Q$ are opposite parabolic
	subgroups. Since $P$ and $Q$ contain $T$, we get $Q = P^-$.
\end{proof}
	
	By \cite[Proposition~14.21]{Bo1991Linear-algebraic-g}
	we have that $PP^-$ is open in $G$ and the product map
	induces an isomorphism of varieties
	\begin{equation}
		\label{eq.productdecomp}	
		R_u(P) \times  (P \cap P^-) \times R_u(P^-)
		\stackrel{\cong}{\longrightarrow} P P^- \, .
	\end{equation}
	In particular, we get the following.
	\begin{lemma}
		\label{lem.dimG_is_dimR_uPminus_plus_dimP}
		We have $\dim G = \dim R_u(P^-) + \dim P$.
	\end{lemma}

\subsection{Dimension of $\UUU_P$ and $R_u(P)$ of a parabolic
subgroup $P$}
We give here a result which estimates the dimension of $\UUU_P$ and
$R_u(P)$ from below for a parabolic subgroup $P$.
The proof is based on the
following fact.
Let $\alpha$ be a simple root
and let $\beta$ be a positive root which is a linear combination
of simple roots different from $\alpha$. If $\alpha$ and $\beta$
are not perpendicular, then $\alpha + \beta$ is a positive root,
by \cite[Lemma~9.4 and Lemma~10.1]{Hu1978Introduction-to-Li}.

\begin{lemma}
	\label{lem.enoughUnipotentsInAParabolic}
	Assume that $G$ is a simple group
	and let $P$ be a parabolic subgroup that contains $B$.
	Then the following holds
	\begin{enumerate}[$i)$]
		\item	If $\rank(G) \geq 3$ and $P \neq B$, then
			$\dim \mathcal{U}_P \geq 2 \rank(G) + 1$.
		\item If $\rank(G) = 2$ and $B \neq P \neq G$, then
			\[
				\dim \UUU_P = \left\{
					\begin{array}{rl}
						4 & \textrm{if $G$ is of type $A_2$,} \\
						5 & \textrm{if $G$ is of type $B_2$,} \\
						7 & \textrm{if $G$ is of type $G_2$.} \\
					\end{array} \right.
			\]
	    	\item If $\rank(G) \geq 2$ and $P \neq G$, then  $\dim R_u(P) \geq 2$.
	\end{enumerate}
\end{lemma}

\begin{proof}
	Assume that $P \neq B$.
	Since $\dim \UUU_P = \dim P - \rank(G)$ we get
	\[
		\dim \UUU_P =
				\dim R_u(B) + (\dim R_u(B) - \dim R_u(P)) \, .
	\]
	Note that $\dim R_u(B)$ is equal to the number of positive roots.
	In a Dynkin diagram the vertices correspond to the simple roots
	and there is one (or more) edges between two simple roots if and only
	if they are not perpendicular. For each pair of non-perpendicular
	simple roots $\alpha$, $\beta$, the sum $\alpha + \beta$ is again
	a (positive) root. Since any Dynkin diagram is a tree,
	the simple roots together with the above sums
	of pairs give $2 \rank(G) - 1$ positive roots.
	
	Assume that $\rank(G) \geq 3$ and $P \neq B$.
	Again, since any Dynkin diagram is a tree,
	one sees that there is a subgraph of the Dynkin diagram of $G$ of the form
	\[
		\xymatrix{ \alpha_1 \ar@{-}[r] & \alpha_2 \ar@{-}[r] & \alpha_3}
	\]
	and $\alpha_1$, $\alpha_3$ are not connected in the Dynkin diagram.
	Hence $\alpha_1 + \alpha_2$ and $\alpha_3$ are not perpendicular
	and thus the sum $\alpha_1 + \alpha_2 + \alpha_3$ is again a positive root.
	Thus we proved $\dim R_u(B) \geq 2 \rank(G)$.
	Since $P$ is not a Borel subgroup, we get
	$\dim R_u(B) -  \dim R_u(P) \geq 1$. These two inequalities yield $i)$.
	
	Assume that $\rank(G) = 2$ and $B \neq P \neq G$. Hence,
	we get $\dim R_u(B) - \dim R_u(P) = 1$,
	by \cite[\S30.2]{Hu1975Linear-algebraic-g}. Considering the
	classification of irreducible root systems of rank two and counting the number
	of positive roots in these root systems yield $ii)$.

	
	Assume that $\rank(G) \geq 2$ and $P \neq G$. Hence,
	there exists a simple root $\alpha$ such that
	$-\alpha$ is not a root of $P$. Since
	$\rank(G) \geq 2$ and since the root system is irreducible,
	there exists a simple root
	$\beta \neq \alpha$ such that $\alpha + \beta$ is a positive root.
	By~\cite[\S30.2]{Hu1975Linear-algebraic-g} it follows
	that $\alpha$ and
	$\alpha + \beta$ are distinct roots of $R_u(P)$, which proves $iii)$.
\end{proof}

\section{Two results on ${\C^+}$-equivariant morphisms of surfaces}\label{sec.C-equivmorphofsurf}

In this section we proof two results on ${\C^+}$-equivariant morphisms
of surfaces
that we use in the proof of Proposition~\ref{prop.key}.
If $S$ is an affine variety with a ${\C^+}$-action, then
we denote by $S \aquot {\C^+}$ the spectrum of the ring of
${\C^+}$-invariant functions on $S$. In general $S \aquot {\C^+}$
is an affine scheme which is not a variety. If the quotient morphism
$S \to S \aquot {\C^+}$ happens to be a principal ${\C^+}$-bundle,
then we denote the algebraic quotient by $S / {\C^+}$.
By Rentschler's Theorem, for a fixed point free action of ${\C^+}$ on the
affine plane $\AA^2$, the algebraic quotient of ${\C^+}$ is a trivial principal
${\C^+}$-bundle over the affine line $\AA \cong \AA^2 / {\C^+}$;
see \cite{Re1968Operations-du-grou}.

\begin{lemma}
	\label{lem.section_birational1}
	Let $S$ be an irreducible, quasi-affine surface and assume that
	${\C^+}$ acts without fixed point on $\AA^2$ and on $S$. If
	$f \colon \AA^2 \to S$ is a dominant and ${\C^+}$-equivariant morphism,
	then there exists a section $X \subseteq \AA^2$ of the algebraic
	quotient $\AA^2 \to \AA^2 / {\C^+}$ such that $f$ induces a
	birational morphism $X \to f(X)$.
\end{lemma}
	
\begin{proof}
	By \cite[Lemma~1]{FaMa1978Quasi-affine-surfa}, there exists
	a ${\C^+}$-invariant open subset $V \subseteq S$ and a smooth affine
	curve $U$ such that $V$ and $U \times {\C^+}$
	are ${\C^+}$-equivariantly isomorphic.
	Hence, $f$ restricts on $f^{-1}(V)$ to a morphism of the form
	\[	
		(f^{-1}(V) / {\C^+}) \times {\C^+} \longrightarrow U \times {\C^+} \, ,
		\quad (x, t) \longmapsto (\bar{f}(x), t + q(x)),
	\]
	where $q$ is a function defined on the curve
	$f^{-1}(V) / {\C^+}$
	and $\bar{f}$ is the morphism $f^{-1}(V) / {\C^+} \to U$
	induced by $f$.
	Therefore, it suffices to find a
	function $p$ on $\AA \cong \AA^2 / {\C^+}$
	(which corresponds to a section
	of $\AA^2 \to \AA^2 / {\C^+}$) such that the morphism
	\begin{equation}
		\label{eq.the_morph}
		f^{-1}(V) / {\C^+} \longrightarrow U \times {\C^+} \, , \quad
		x \longmapsto (\bar{f}(x), p(x)+q(x))
	\end{equation}
	is birational onto its image. After shrinking $V$, we can assume
	that $\bar{f}$ is finite and \'etale. Fix $u_0 \in U$. One can choose $p$
	such that the points
	\[
		(u_0, p(x_1) + q(x_1)) \, , \ldots , (u_0, p(x_k) + q(x_k))
	\]
	are all distinct, where $x_1, \ldots, x_k$ denote the elements
	of the fiber of $\bar{f}$ over $u_0$.
	The same is still true for elements in a neighbourhood
	of $u_0$, as one can see by choosing an \'etale neighbourhood
	of $u_0$ in $U$ which trivializes
	$\bar{f}$ at $u_0$ with respect to the \'etale topology; see
	\cite[Chp. I, Corollary~3.12]{Mi1980Etale-cohomology}.
	Hence \eqref{eq.the_morph} is injective on an open subset
	of $f^{-1}(V) / {\C^+}$, i.e.~it is birational onto its image.
\end{proof}

\begin{lemma}
	\label{lem.section_embedd}
	Let $S$ be an irreducible, quasi-affine surface and assume that
	${\C^+}$ acts without fixed point on $\AA^2$ and on $S$. If
	$f \colon \AA^2 \to S$ is a {${\C^+}$-equivariant}
	birational morphism,
	then there exists a section $X \subseteq \AA^2$ of
	$\AA^2 \to \AA^2 / {\C^+}$ such that $f$ restricts to an
	embedding on $X$.
\end{lemma}	

\begin{proof}
	We identify $\AA^2$ with $\AA \times {\C^+}$
	and consider it as a trivial principal ${\C^+}$-bundle over $\AA$.
	For $\alpha \in {\C^*}$ let
	\[
		Z_{\alpha} = \{ \,Ê(x, \alpha x) \ | \ x \in \AA \, \}
		\subseteq \AA \times {\C^+} \, .
	\]
	We claim, that for generic $\alpha \in {\C^*}$
	the map $f$ restricts to an embedding on
	$Z_{\alpha}$. In other words, we claim that $f$ restricted to $Z_\alpha$
	is injective and immersive for generic $\alpha$ (the properness
	is then automatically satisfied, since $Z_\alpha \cong \AA$).
	The claim then implies the statement of the lemma.
	
	Let us first prove injectivity.
	Since $f$ is ${\C^+}$-equivariant
	and birational, there exists a ${\C^+}$-invariant open subset of
	$\AA \times {\C^+}$
	that is mapped isomorphically onto a ${\C^+}$-invariant open subset of $S$.
	Since ${\C^+}$ acts without fixed point, it follows that there
	are only finitely many ${\C^+}$-orbits $F$ in $S$ such that
	the inverse image $f^{-1}(F)$ consists of more than one ${\C^+}$-orbit.
	Thus, it is enough to show that $f$ is injective on $f^{-1}(F) \cap Z_\alpha$
	for fixed $F$ and generic $\alpha$ in $\C^\ast$.
	So let $F \subseteq S$ be a ${\C^+}$-orbit
	such that there exist $k > 1$ and
	distinct $x_1, \ldots, x_k \in \AA$ such that $f^{-1}(F)$
	is the union of the lines $L_i = \{ x_i \} \times {\C^+}$,
	$i = 1, \ldots, k$. Moreover, there exist
	$\beta_i \in {\C^+}$ such that $f |_{L_i} \colon L_i \to F$ is given by
	$t \mapsto t + \beta_i$, where we have identified the orbit $F$ with ${\C^+}$.
	Injectivity of $f$ on $f^{-1}(F) \cap Z_{\alpha}$ for generic $\alpha$ follows,
	since for generic $\alpha$ we have
	\[
		\alpha x_i + \beta_i
		\neq \alpha x_j + \beta_j
		\quad
		\textrm{for all $i \neq j$} \, .
	\]
	
	Let us prove immersivity. As already mentioned, there
	exists an open ${\C^+}$-invariant subset
	$U \subseteq \AA \times {\C^+}$
	such that $f$ restricts to an open injective immersion on $U$.
	Let $x_0 \in \AA$ such that $\{ x_0 \} \times {\C^+}$ lies in
	the complement of $U$ in $\AA \times {\C^+}$.
	Since there are only finitely many such $x_0 \in \AA$, it is enough to show
	that for generic $\alpha \in {\C^*}$ the restriction
	$f |_{Z_{\alpha}}$ is immersive in the point $(x_0, \alpha x_0)$.
	Since ${\C^+}$ acts without fixed
	point on $S$ and since $f$ is ${\C^+}$-equivariant, the kernel of the
	differential of $f$ is at most one-dimensional in every point of
	$\AA \times {\C^+}$.
	Since the tangent direction of $Z_{\alpha}$
	in the point $(x_0, \alpha x_0)$ is given by
	$(1, \alpha)$, we proved that $f |_{Z_{\alpha}}$ is immersive
	in $(x_0, \alpha x_0)$ for generic $\alpha$.
	This proves the immersivity.
\end{proof}

\newcommand{\etalchar}[1]{$^{#1}$}
\providecommand{\bysame}{\leavevmode\hbox to3em{\hrulefill}\thinspace}
\providecommand{\MR}{\relax\ifhmode\unskip\space\fi MR }
\providecommand{\MRhref}[2]{%
  \href{http://www.ams.org/mathscinet-getitem?mr=#1}{#2}
}


\begin{thebibliography}{AFK{\etalchar{+}}13}

\bibitem[AM75]{AbMo1975Embeddings-of-the-}
Shreeram~S. Abhyankar and Tzuong~Tsieng Moh, \emph{Embeddings of the line in
  the plane}, J. Reine Angew. Math. \textbf{276} (1975), 148--166.

\bibitem[AFK{\etalchar{+}}13]{ArFlKa2013Flexible-varieties}
Ivan~Arzhantsev, Hubert~Flenner, Shulim~Kaliman,
Frank~Kutzschebauch, and Mikhail~Zaidenberg,
\emph{Flexible varieties and automorphism groups},
Duke Math. J. \textbf{162} (2013), no.~4, 767--823.

\bibitem[AM11]{AsMo2011Smooth-varieties-u}
Aravind Asok and Fabien Morel, \emph{Smooth varieties up to {$\Bbb
  A^1$}-homotopy and algebraic {$h$}-cobordisms}, Adv. Math. \textbf{227}
  (2011), no.~5, 1990--2058.

\bibitem[BCW77]{BaCoWr1976Locally-polynomial}
H.~Bass, E.~H. Connell, and D.~L. Wright, \emph{Locally polynomial algebras are
  symmetric algebras}, Invent. Math. \textbf{38} (1976/77), no.~3, 279--299.

\bibitem[Bor91]{Bo1991Linear-algebraic-g}
Armand Borel, \emph{Linear algebraic groups}, second ed., Graduate Texts in
  Mathematics, vol. 126, Springer-Verlag, New York, 1991.

\bibitem[BL03]{BrLa2003A-geometric-approa}
M.~Brion and V.~Lakshmibai, \emph{A geometric approach to standard monomial
  theory}, Represent. Theory \textbf{7} (2003), 651--680.

\bibitem[Che05]{Ch2005Classification-des}
Claude Chevalley, \emph{Classification des groupes alg{\'e}briques
  semi-simples}, Springer-Verlag, Berlin, 2005, Collected works. Vol. 3, Edited
  and with a preface by P. Cartier, With the collaboration of Cartier, A.
  Grothendieck and M. Lazard.

\bibitem[Cra04]{Cr2004On-the-AK-invarian}
Anthony~J. Crachiola, \emph{On the {AK} invariant of certain domains}, ProQuest
  LLC, Ann Arbor, MI, 2004, Thesis (Ph.D.)--Wayne State University.

\bibitem[DD16]{DeDu2016Affine-Lines-in-th}
Julie Decaup and Adrien Dubouloz, \emph{Affine lines in the complement of a
  smooth plane conic}, in preparation, September 2016.


\bibitem[Eis95]{Ei1995Commutative-algebr}
David Eisenbud, \emph{Commutative algebra}, Graduate Texts in Mathematics, vol.
  150, Springer-Verlag, New York, 1995, With a view toward algebraic geometry.

\bibitem[FM78]{FaMa1978Quasi-affine-surfa}
Amassa Fauntleroy and Andy~R. Magid, \emph{Quasi-affine surfaces with
  {$G_{a}$}-actions}, Proc. Amer. Math. Soc. \textbf{68} (1978), no.~3,
  265--270.

\bibitem[Fre57]{Fr1957Cohomologie-non-ab}
Jean Frenkel, \emph{Cohomologie non ab{\'e}lienne et espaces fibr{\'e}s}, Bull.
  Soc. Math. France \textbf{85} (1957), 135--220.

\bibitem[Gro61]{Gr1961Elements-de-geomet-III_1}
A.~Grothendieck, \emph{\'{E}l{\'e}ments de g{\'e}om{\'e}trie alg{\'e}brique.
  {III}. \'{E}tude cohomologique des faisceaux coh{\'e}rents. {I}}, Inst.
  Hautes {\'E}tudes Sci. Publ. Math. (1961), no.~11, 167.

\bibitem[Gro66]{Gr1966Elements-de-geomet-IV_3}
\bysame, \emph{\'{E}l{\'e}ments de g{\'e}om{\'e}trie alg{\'e}brique.
  {IV}. \'{E}tude locale des sch{\'e}mas et des morphismes de sch{\'e}mas.
  {III}}, Inst. Hautes {\'E}tudes Sci. Publ. Math. (1966), no.~28, 255.

\bibitem[GR04]{GrRa2004Revetements-etales}
Alexander Grothendieck and Michele Raynaud, \emph{{Rev{\^e}tements {\'e}tales
  et groupe fondamental (SGA 1)}}, 2004, \url{http://arxiv.org/abs/math/0206203}.

\bibitem[GM92]{GuMi1992Affine-lines-on-lo}
R.~V. Gurjar and M.~Miyanishi, \emph{Affine lines on logarithmic {${\bf
  Q}$}-homology planes}, Math. Ann. \textbf{294} (1992), no.~3, 463--482.

\bibitem[Har77]{Ha1977Algebraic-geometry}
Robin Hartshorne, \emph{Algebraic geometry}, Graduate Texts in Mathematics,
  vol.~52, Springer-Verlag, New York, 1977.

\bibitem[Hum95]{Hu1995Conjugacy-classes-}
		James~E. Humphreys, \emph{Conjugacy classes in
		semisimple algebraic groups},
  		Mathematical Surveys and Monographs,
		vol.~43, American Mathematical Society,
  		Providence, RI, 1995.

\bibitem[Hum78]{Hu1978Introduction-to-Li}
		\bysame, \emph{Introduction to {L}ie
		algebras and representation theory},
		Graduate Texts in Mathematics, vol.~9, Springer-Verlag, New
  		York-Berlin, 1978, Second printing, revised.

\bibitem[Hum75]{Hu1975Linear-algebraic-g}
\bysame, \emph{Linear algebraic groups}, Springer-Verlag, New York,
  1975, Graduate Texts in Mathematics, No. 21.

\bibitem[Jel87]{Je1987The-extension-of-r}
Zbigniew Jelonek, \emph{The extension of regular and rational embeddings},
  Math. Ann. \textbf{277} (1987), no.~1, 113--120.

\bibitem[Kal91]{Ka1991Extensions-of-isom}
Shulim Kaliman, \emph{Extensions of isomorphisms between affine algebraic
  subvarieties of {$k^n$} to automorphisms of {$k^n$}}, Proc. Amer. Math. Soc.
  \textbf{113} (1991), no.~2, 325--334.

\bibitem[Kal94]{Ka1994Exotic-analytic-st}
\bysame, \emph{Exotic analytic structures and {E}isenman intrinsic
  measures}, Israel J. Math. \textbf{88} (1994), no.~1-3, 411--423.

\bibitem[KW85]{KaWr1985Flat-families-of-a}
T.~Kambayashi and David Wright, \emph{Flat families of affine lines are
  affine-line bundles}, Illinois J. Math. \textbf{29} (1985), no.~4, 672--681.

\bibitem[Kle74]{Kl1974The-transversality}
Steven~L. Kleiman, \emph{The transversality of a general translate}, Compositio
  Math. \textbf{28} (1974), 287--297.

\bibitem[Kol96]{Ko1996Rational-curves-on}
J{{\'a}}nos Koll{{\'a}}r, \emph{Rational curves on algebraic varieties},
  Ergebnisse der Mathematik und ihrer Grenzgebiete. 3. Folge. A Series of
  Modern Surveys in Mathematics, vol.~32, Springer-Verlag,
  Berlin, 1996.

\bibitem[Kra96]{Kr1996Challenging-proble}
Hanspeter Kraft, \emph{Challenging problems on affine {$n$}-space},
  Ast{\'e}risque (1996), no.~237, Exp.\ No.\ 802, 5, 295--317, S{{\'e}}minaire
  Bourbaki, Vol. 1994/95.

\bibitem[KR14]{KrRu2014Families-of-group-}
Hanspeter Kraft and Peter Russell, \emph{Families of group actions, generic
  isotriviality, and linearization}, Transform. Groups \textbf{19} (2014),
  no.~3, 779--792.

\bibitem[ML01]{Ma2001On-the-group-of-au}
L.~Makar-Limanov, \emph{On the group of automorphisms of a surface
  {$x^ny=P(z)$}}, Israel J. Math. \textbf{121} (2001), 113--123.

\bibitem[Mat86]{Ma1986Commutative-ring-t}
Hideyuki Matsumura, \emph{Commutative ring theory}, Cambridge Studies in
  Advanced Mathematics, vol.~8, Cambridge University Press, Cambridge, 1986.

\bibitem[Mil80]{Mi1980Etale-cohomology}
James~S. Milne, \emph{\'{E}tale cohomology}, Princeton Mathematical Series,
  vol.~33, Princeton University Press, Princeton, N.J., 1980.

\bibitem[Miy75]{Mi1975An-algebraic-chara}
Masayoshi Miyanishi, \emph{An algebraic characterization of the affine plane},
  J. Math. Kyoto Univ. \textbf{15} (1975), 169--184.

\bibitem[Miy84]{Mi1984An-algebro-topolog}
\bysame, \emph{An algebro-topological characterization of the
  affine space of dimension three}, Amer. J. Math. \textbf{106} (1984), no.~6,
  1469--1485.

\bibitem[MS80]{MiSu1980Affine-surfaces-co}
Masayoshi Miyanishi and Tohru Sugie, \emph{Affine surfaces containing
  cylinderlike open sets}, J. Math. Kyoto Univ. \textbf{20} (1980), no.~1,
  11--42.

\bibitem[OV90]{OnVi1990Lie-groups-and-alg}
A.~L. Onishchik and {{\`E}}.~B. Vinberg, \emph{Lie groups and algebraic
  groups}, Springer Series in Soviet Mathematics, Springer-Verlag, Berlin,
  1990, Translated from the Russian and with a preface by D. A. Leites.

\bibitem[OY82]{OnYo1982On-Noetherian-subr}
Nobuharu Onoda and Ken-ichi Yoshida, \emph{On {N}oetherian subrings of an
  affine domain}, Hiroshima Math. J. \textbf{12} (1982), no.~2, 377--384.

\bibitem[RR84]{RaRa1984Principal-bundles-}
M.~S. Raghunathan and A.~Ramanathan, \emph{Principal bundles on the affine
  line}, Proc. Indian Acad. Sci. Math. Sci. \textbf{93} (1984), no.~2-3,
  137--145.

\bibitem[RR85]{RaRa1985Projective-normali}
S.~Ramanan and A.~Ramanathan, \emph{Projective normality of flag varieties and
  {S}chubert varieties}, Invent. Math. \textbf{79} (1985), no.~2, 217--224.

\bibitem[Ram83]{Ra1983Deformations-of-pr}
A.~Ramanathan, \emph{Deformations of principal bundles on the projective line},
  Invent. Math. \textbf{71} (1983), no.~1, 165--191.

\bibitem[Ram85]{Ra1985Schubert-varieties}
\bysame, \emph{Schubert varieties are arithmetically {C}ohen-{M}acaulay},
  Invent. Math. \textbf{80} (1985), no.~2, 283--294.

\bibitem[Ren68]{Re1968Operations-du-grou}
Rudolf Rentschler, \emph{Op{\'e}rations du groupe additif sur le plan affine},
  C. R. Acad. Sci. Paris S{\'e}r. A-B \textbf{267} (1968), A384--A387.

\bibitem[Ric92]{Ri1992Intersections-of-d}
R.~W. Richardson, \emph{Intersections of double cosets in algebraic groups},
  Indag. Math. (N.S.) \textbf{3} (1992), no.~1, 69--77.

\bibitem[Ros61]{Ro1961Toroidal-algebraic}
Maxwell Rosenlicht, \emph{Toroidal algebraic groups}, Proc. Amer. Math. Soc.
  \textbf{12} (1961), 984--988.

\bibitem[Ser58]{Se1958Espaces-fibres-}
		Jean-Pierre~Serre, \emph{Espaces fibr\'es alg\'ebriques},
		In: Annequx de Chow et applications, Seminaire Chevalley, 1958.

\bibitem[Ser94]{Se1994Cohomologie-galois}
Jean-Pierre Serre, \emph{Cohomologie galoisienne}, fifth ed., Lecture Notes in
  Mathematics, vol.~5, Springer-Verlag, Berlin, 1994.

\bibitem[Sha92]{Sh1992Polynomial-represe}
Anant~R. Shastri, \emph{Polynomial representations of knots}, Tohoku Math. J.
  (2) \textbf{44} (1992), no.~1, 11--17.

\bibitem[Spr09]{Sp2009Linear-algebraic-g}
T.~A. Springer, \emph{Linear algebraic groups}, second ed., Modern
  Birkh{\"a}user Classics, Birkh{\"a}user Boston, Inc., Boston, MA, 2009.

\bibitem[Sri91]{Sr1991On-the-embedding-d}
V.~Srinivas, \emph{On the embedding dimension of an affine variety}, Math. Ann.
  \textbf{289} (1991), no.~1, 125--132.

\bibitem[Sta15]{St2015Algebraic-Embeddin}
Immanuel Stampfli, \emph{Algebraic embeddings of $\mathbb{C}$ into
$\mathrm{SL}_n(\mathbb{C})$}, Transformation Groups (2015).

\bibitem[Ste65]{St1965Regular-elements-o}
Robert Steinberg, \emph{Regular elements of semisimple algebraic groups}, Inst.
  Hautes {\'E}tudes Sci. Publ. Math. (1965), no.~25, 49--80.

\bibitem[Ste76]{St1976On-the-desingulari}
\bysame, \emph{On the Desingularization of the Unipotent Variety}, Invent.
Math. \textbf{36} (1976), 209--224.

\bibitem[Suz74]{Su1974Proprietes-topolog}
Masakazu Suzuki, \emph{Propri\'et\'es topologiques des polyn\^omes de deux
  variables complexes, et automorphismes alg\'ebriques de l'espace {${\bf
  C}^{2}$}}, J. Math. Soc. Japan \textbf{26} (1974), 241--257.

\bibitem[Tim11]{Ti2011Homogeneous-spaces}
Dmitry~A. Timashev, \emph{Homogeneous spaces and equivariant embeddings},
  Encyclopaedia of Mathematical Sciences, vol. 138, Springer, Heidelberg, 2011,
  Invariant Theory and Algebraic Transformation Groups, 8.

\bibitem[tDP90]{DiPe1990Contractible-affin}
Tammo tom Dieck and Ted Petrie, \emph{Contractible affine surfaces of {K}odaira
  dimension one}, Japan. J. Math. (N.S.) \textbf{16} (1990), no.~1, 147--169.

\bibitem[vdE04]{Es2004Around-the-Abhyank}
Arno van~den Essen, \emph{Around the {A}bhyankar-{M}oh theorem}, Algebra,
  arithmetic and geometry with applications ({W}est {L}afayette, {IN}, 2000),
  Springer, Berlin, 2004, pp.~283--294.
		
\end{thebibliography}
\end{document}